\documentclass{amsart}
\usepackage{tikz}
\usetikzlibrary{arrows,matrix}
\usepackage[all]{xy}

\begin{document}
\title[Sol. to arith. diff. eq. in alg. closed fields]{Solutions to arithmetic differential equations\\
in algebraically closed fields}

\author{Alexandru Buium}
\address{Department of Mathematics and Statistics,
University of New Mexico, Albuquerque, NM 87131, USA}
\email{buium@math.unm.edu} 

\author{Lance Edward Miller}
\address{Department of Mathematical Sciences,  309 SCEN,
University of Arkansas, 
Fayetteville, AR 72701}
\email{lem016@uark.edu}

\def \d{\delta}
\def \ra{\rightarrow}
\def \bZ{{\mathbb Z}}
\def \cO{{\mathcal O}}
\newcommand{\Hom}{\operatorname{Hom}}

\newcommand{\OCp}{ {\mathbb C}_p^{\circ}  }

\newtheorem{THM}{{\!}}[section]
\newtheorem{THMX}{{\!}}
\renewcommand{\theTHMX}{}
\newtheorem{theorem}{Theorem}[section]
\newtheorem{corollary}[theorem]{Corollary}
\newtheorem{lemma}[theorem]{Lemma}
\newtheorem{proposition}[theorem]{Proposition}
\theoremstyle{definition}
\newtheorem{definition}[theorem]{Definition}
\theoremstyle{remark}
\newtheorem{remark}[theorem]{Remark}
\newtheorem{example}[theorem]{\bf Example}
\numberwithin{equation}{section}
\maketitle

\begin{abstract}
We prove that the ``main examples" in the theory of arithmetic differential equations \cite{book}
possess a remarkable ``total differential overconvergence property". This allows one to consider solutions to these equations with coordinates in  algebraically closed fields.
\end{abstract}

\section{Introduction}

Arithmetic differential equations \cite{siegel, char, difmod, book} are an analogue of usual differential equations in which functions  are replaced by numbers  and the derivation
operator  is replaced by a {\it $p$-derivation} (Fermat quotient operator) acting on a ring of numbers. Typically, one takes the ring of numbers
 to be $R:=\widehat{\bZ_p^{\text{ur}}}$ and then
 arithmetic differential equations can be viewed as certain functions
\begin{equation}
\label{voxvox}
f:V(R)\ra R\end{equation}
defined on sets of $R$-points of various smooth schemes $V$ over $R$; 
these functions, referred to as {\it $\d$-functions} are given by  limits of polynomials in affine coordinates and their iterated $p$-derivatives up to a certain fixed order.
The preimage $f^{-1}(0)$ can be interpreted as the set of {\it unramified} solutions to $f$.

In this paper we introduce a class of functions $V(R^{\text{alg}})\ra K^{\text{alg}}$, referred to as $\d^{\text{alg}}$-{\it functions}, where $K^{\text{alg}}$ is the algebraic closure of 
$K:=R[1/p]$ and $R^{\text{alg}}$ is the valuation ring of $K^{\text{alg}}$. 
As with $\d$-functions, 
$\d^{\text{alg}}$-functions 
have the following ``analytic continuation" property which makes them behave like ``global objects": if a $\d^{\text{alg}}$-function $g$ vanishes on a $p$-adic ball in $V(R^{\text{alg}})$ and $V$ has a Zariski connected reduction mod $p$ then $g$ vanishes everywhere. 
We will also distinguish among $\d^{\text{alg}}$-functions a subclass of functions which we call {\it tempered} $\d^{\text{alg}}$-functions. Tempered 
$\d^{\text{alg}}$-functions are distinguished among $\d^{\text{alg}}$-functions by asking that 
 the $p$-adic absolute value  of the value of the function at any point be bounded by a constant times a power of the ramification index of that point.

The aim of this paper is to show that the ``main examples" of arithmetic differential equations
\ref{voxvox} in the theory can be uniquely extended to tempered $\d^{\text{alg}}$-functions
\begin{equation}
\label{voxvoxvox}
f^{\text{alg}}:V(R^{\text{alg}})\ra K^{\text{alg}}.\end{equation}
The preimage $(f^{\text{alg}})^{-1}(0)$ can be interpreted as 
the set of {\it algebraic} (hence arbitrarly ramified)  solutions to $f$. The association $f\mapsto f^{\text{alg}}$ will commute with the ring operations and $p$-derivations and hence will preserve all ``differential algebraic properties". In particular, it will send {\it $\d$-characters} \cite{char} into $\d$-characters and   {\it isogeny covariant $\d$-modular forms} \cite{difmod,book} into isogeny covariant $\d$-modular forms. 

Now $K^{\text{alg}}$ is, of course, dense in the $p$-adic complex field ${\mathbb C}_p$ and one can ask if our functions $f^{\text{alg}}$ in \ref{voxvoxvox} can be further extended to continuous functions
\begin{equation}
\label{voxodrom}
f^{{\mathbb C}_p}:V(\OCp)\ra {\mathbb C}_p,\end{equation}
where $\OCp$ is the valuation ring of ${\mathbb C}_p$. 
This does not seem to be automatic for arbitrary $\d^{\text{alg}}$-functions (or even for the tempered ones) but we will prove this is the case for the
 $\d^{\text{alg}}$-functions  \ref{voxvoxvox} arising from some of the main examples of $\d$-functions \ref{voxvox} in the theory.  The extension \ref{voxodrom} of a $\d^{\text{alg}}$-function, if it exist, is unique and 
 will be referred to as a $\d^{{\mathbb C}_p}$-function.
  Except in trivial cases, $\d^{{\mathbb C}_p}$-functions are not locally analytic, in particular they are not rigid analytic. But, of course, $\d^{{\mathbb C}_p}$-functions  have the same ``analytic continuation" property as before: if they vanish on a ball then they vanish everywhere.

Here is the heuristic behind the theory in the present paper.
Roughly speaking  what happens when one passes from $R$ to  $R^{\text{alg}}$ 
is that $p$-derivatives of numbers in ramified extensions of $R$ ``acquire denominators"; this phenomenon is entirely similar to what happens in usual calculus where the derivative of a ramified algebraic function  ``acquires denominators"; e.g.
\begin{equation}
\label{ubu}
\frac{d}{dx}(\sqrt{x})=\frac{1}{2\sqrt{x}}.\end{equation}
The presence of such denominators in the arithmetic setting potentially introduces divergences in the  $p$-adic series of the theory and should be, in principle, an obstacle to developing the theory  in the arbitrarily ramified case.
The main point of the present paper is to show that this potential obstacle can be overcome
for  the ``main examples" of the theory. The main idea is to ``rescale"  the $p$-derivation every time we increase the ramification; naively speaking this corresponds, in the case of classical calculus, to passing from the equality \ref{ubu} to the equality
\begin{equation}
\label{uburoi}
\frac{d}{d\sqrt{x}}(\sqrt{x})=1.\end{equation}
Rescaling the $p$-derivation, in the arithmetic setting, will lead us to revisit the concept of {\it differential overconvergence} introduced in
\cite{over}; the main point then will be to show that one can  
drop, ``everywhere" in \cite{over}, the requirement of ``bounded ramification." 
That this is possible is indicated already by  a construction of aribitrarly ramified ``$\d$-characters" in \cite{newforms}. On the other hand it would be interesting to see if 
the  ``unbounded ramification" theory of the present paper can be made to interact with 
 the Borger-Saha construction of isocrystals in  \cite{borgersaha}.

For the convenience of the reader the main concepts and results of the present paper will be presented without any reference to \cite{char, difmod, book, over, newforms}; some references to these papers will be made, however, in our proofs.

\bigskip

{\bf Acknowledgement}. The first author 
is grateful to Max Planck Institute for Mathematics in Bonn for its hospitality and financial support; to the Simons Foundation  for support through awards 311773 and
615356; and to  A. Saha for enlightening conversations.

\bigskip

\bigskip

\section{Main concepts and results}

\subsection{Basic fields and rings}\label{meetwad}
Throughout  this paper $p\in \bZ$ is a fixed prime which we assume for technical reasons to be at least $5$. Consider the diagram of valued  fields:

\begin{center}
\begin{tikzpicture}

\node (LT) at (-1,0) {$ {\mathbb Q}_p^{\text{ur}} $};
\node (RT) at (1,0){$ {\mathbb Q}_p^{\text{alg}} $};
\node (LB) at (-1,-1){$ K $};
\node (RB) at (1,-1){$ K^{\text{alg}} $};

\draw[->] (LT) edge (RT);
\draw[->] (LT) edge (LB);
\draw[->] (RT) edge (RB);
\draw[->] (LB) edge (RB);

\end{tikzpicture}
\end{center}


where ${\mathbb Q}_p^{\text{alg}}$ is an algebraic closure of ${\mathbb Q}_p$,
${\mathbb Q}_p^{\text{ur}}$  is the maximum unramified extension of ${\mathbb Q}_p$
inside ${\mathbb Q}_p^{\text{alg}}$, $K$ is the metric completion of ${\mathbb Q}_p^{\text{ur}}$ and $K^{\text{alg}}$ is the algebraic closure of $K$ in the metric completion ${\mathbb C}_p$
of ${\mathbb Q}_p^{\text{alg}}$. Of course $K^{\text{alg}}$ coincides with the maximum totally ramified extension $K^{\text{tot}}$ of $K$. Also,
by Krasner's Lemma,  we have $K^{\text{alg}}:=K{\mathbb Q}_p^{\text{alg}}$,
the compositum of $K$ and ${\mathbb Q}_p^{\text{alg}}$; cf.  \cite[pg. 149, Prop. 5]{bosch}.
We denote by $v_p$ the valuation on ${\mathbb C}_p$ normalized by the condition $v_p(p)=1$ and for $a\in {\mathbb C}_p$ we denote by $|a|_p=p^{-v_p(a)}$ the absolute value.
We denote by
\begin{equation}
\label{aleluia}
{\mathbb Z}_p^{\text{ur}}, {\mathbb Z}_p^{\text{alg}}, R, R^{\text{alg}}, \OCp
\end{equation}
the valuation rings of 
\begin{equation}
{\mathbb Q}_p^{\text{ur}}, {\mathbb Q}_p^{\text{alg}}, K, K^{\text{alg}}, {\mathbb C}_p\end{equation}
respectively.
In particular, 
$$R:=\widehat{{\mathbb Z}_p^{\text{ur}}},\ \ \ \OCp=\widehat{R^{\text{alg}}}.$$
 Here and later the symbol $\widehat{\ }$
always denotes $p$-adic completion of rings or schemes.
  We denote by $k$ the common residue field of the rings \ref{aleluia}; so $k$ is  an algebraic closure of ${\mathbb F}_p$.
The ring $R$ possesses  a unique automorphism $\phi:R\ra R$ whose reduction mod $p$ is the $p$-power Frobenius on $k$. Denote by $\phi:K\ra K$ the induced automorphism of $K$ and
 fix, throughout the paper, an  automorphism (which we still denote by $\phi$) of $K^{\text{alg}}$
extending the automorphism $\phi$ of $K$. 
Since $K$ is complete $v_p$ on $K^{\text{alg}}$ is the unique valuation extending $v_p$ on $K$  so $v_p\circ \phi=v_p$ on $K^{\text{alg}}$ and hence the automorphism $\phi$ of $K^{\text{alg}}$  induces an automorphism (still denoted by $\phi$) of $R^{\text{alg}}$.

Denote by $\Pi$ the set of all roots $\pi$ in ${\mathbb Q}_p^{\text{alg}}$ of Eisenstein polynomials with coefficients in $\bZ_p^{\text{ur}}$ having the property that the extension
${\mathbb Q}_p(\pi)/{\mathbb Q}_p$ is Galois. 
So ${\mathbb Q}_p^{\text{alg}}$ is obtained from ${\mathbb Q}_p^{\text{ur}}$ by adjoining the set $\Pi$.
For any $\pi\in \Pi$ write $K_{\pi}=K(\pi)$, let 
$R_{\pi}=R[\pi]$ be the valuation ring of $K_{\pi}$, and set $e(\pi):=[K_{\pi}:K]$,
the degree (and also the ramification index) of $\pi$ over $K$. Also we write $\pi|\pi_0$ if and only if $K_{\pi_0}\subset K_{\pi}$. 
Note that, by above discussion, we have
\begin{equation}
\label{caldura}
K^{\text{alg}}=\bigcup_{\pi\in \Pi} K_{\pi}\end{equation}
and the maximal ideal of $R^{\text{alg}}$ is generated by $\Pi$. 
Clearly for $\pi\in \Pi$ the field $K_{\pi}$ is mapped into itself by $\phi:K^{\text{alg}}\ra K^{\text{alg}}$ and we have an induced automorphism
$$\phi:R_{\pi}\ra R_{\pi}$$ inducing the $p$-power Frobenius on $R_{\pi}/\pi R_{\pi}=k$.

 \subsection{$\d$-functions and $\d_{\pi}$-functions \cite{char,difmod,book}}
For $\pi\in \Pi$ recall the {\it Fermat quotient operator} \cite{char} 
 $$\d_{\pi}:R_{\pi}\ra R_{\pi},\ \ \ 
\d_{\pi} x  :=  \frac{\phi(x)-x^p}{\pi}, \ \ \  x\in R_{\pi}.$$
This is a special case of what in \cite{char, analogues} was called a $\pi$-derivation, cf. our discussion in Section \ref{ark} below. The setting in \cite{char} was such that $\pi$ was assumed to be totally ramified  over ${\mathbb Q}_p$ but we will not impose this condition here.
In particular, for $\pi=p$ we have a Fermat quotient operator denoted simply by
$$\d:=\d_p:R\ra R.$$

\begin{definition} \cite[pg. 317]{char}
\label{charrrr}
Let $V$ be an affine smooth scheme over $R$ and fix a closed embedding $V \subset {\mathbb A}^d$ over $R$.
A function 
$$f_{\pi}:V(R_{\pi})\ra R_{\pi}$$ is called a $\d_{\pi}$-{\it function} of order $r\geq 0$ on $V$ if there exists a restricted power series $F_{\pi}\in R_{\pi}[x,x',...,x^{(r)}]^{\widehat{\ }}$, where $x,x',...,x^{(r)}$ are $d$-tuples of variables,  such that
\begin{equation}
f_{\pi}(a)=F_{\pi}(a,\d_{\pi} a,...,\d_{\pi}^r a), \ \ a \in V(R_{\pi})\subset R_{\pi}^d.\end{equation}
If $V$ is not necessarily affine $f_{\pi}$ is called a {\it $\d_{\pi}$-function} if its restriction to the $R_{\pi}$-points of any affine subset of $V$ is a $\d_{\pi}$-function.
If $V=G$ is a group scheme  a {\it $\d_{\pi}$-character} of $G$ is a $\d_{\pi}$-function $G(R_{\pi})\ra R_{\pi}$ which is also a group homomorphism into the additive group of $R_{\pi}$.  \end{definition}

The concept above is independent of the embedding. We denote by $\cO^r_{\pi}(V)$ the ring of $\d_{\pi}$-functions on $V$ of order $r\geq 0$.
The association $U\mapsto \cO_{\pi}^r(U)$ for $U\subset V$ Zariski open is a sheaf on $V$. For $\pi=p$,   $\d_p$-functions will be referred to as {\it $\d$-functions} and we write $f=f_p$,
$$f:V(R)\ra R.$$
 Also  we write $\cO^r(V):=\cO^r_p(V)$. The $\d_p$-characters will be simply called $\d$-{\it characters}. 

\subsection{Total $\d$-overconvergence} Recall from \cite[Def. 2.4]{over} the following:

\begin{definition}
Let $f\in \cO^r(V)$, $f:V(R)\ra R$, be a $\d$-function on $V$ of order $r$ and let $\pi\in \Pi$.  We will say that $f$ is $\d_{\pi}$-{\it overconvergent} if there exists an integer $\nu\geq 0$ and a $\d_{\pi}$-function $f_{\pi,\nu}$ on $V$ of order $r$  making the diagram below commutative:
\begin{equation}
\label{pantof2}
\begin{tikzpicture}[baseline=(current bounding box.center)]

\node (LT) at (-1,0) {$ V(R) $};
\node (RT) at (1,0){$ R $};
\node (LB) at (-1,-1){$ V(R_\pi) $};
\node (RB) at (1,-1){$ R_{\pi} $};

\draw[->] (LT) edge node[above]{$p^\nu f$}(RT);
\draw[->] (LT) edge  (LB);
\draw[->] (RT) edge (RB);
\draw[->] (LB) edge node[above]{$f_{\pi,\nu}$} (RB);

\end{tikzpicture}
\end{equation}

where the vertical arrows  are the natural inclusions.
The smallest $\nu$ for which there is a $f_{\pi,\nu}$ as above is called the {\it polar order} of $f$.
\end{definition}

\begin{remark}\label{bubule} It will be checked later, in  Remark \ref{bibi}, that for a given $f$ and $\nu$, the map $f_{\pi,\nu}$ in Diagram \ref{pantof2} is unique. \end{remark}

\begin{remark} \label{sirpac}
By the uniqueness property in Remark \ref{bubule} we have that if 
 \ref{pantof2} holds for some $\nu$ and a necessarily unique $f_{\pi,\nu}$ it also holds for $
 \nu+1$ with a necessarily unique $f_{\pi,\nu+1}$. We have an equality of maps
$$p f_{\pi,\nu}=f_{\pi,\nu+1}:V(R_{\pi})\ra R_{\pi}$$ 
hence an equality of maps
$$p^{-\nu} f_{\pi,\nu}=p^{-(\nu+1)}f_{\pi,\nu+1}:V(R_{\pi})\ra K_{\pi}.$$
We set
$$p^{-*}f_{\pi,*}:=p^{-\nu} f_{\pi,\nu}:V(R_{\pi})\ra p^{-\nu}R_{\pi}\subset K_{\pi},\ \ \ \text{for}\ \ \ \nu>>0;$$
hence we have a commutative diagram
\begin{equation}
\label{pantof222}
\begin{tikzpicture}[baseline=(current bounding box.center)]

\node (LT) at (-1,0) {$ V(R) $};
\node (RT) at (2,0){$ R $};
\node (LB) at (-1,-1){$ V(R_\pi) $};
\node (RB) at (2,-1){$ K_{\pi} $};

\draw[->] (LT) edge node[above]{$f$}(RT);
\draw[->] (LT) edge  (LB);
\draw[->] (RT) edge (RB);
\draw[->] (LB) edge node[above]{$p^{-*}f_{\pi,*}$} (RB);

\end{tikzpicture}
\end{equation}

Moreover if $\pi|\pi_0$ and $f$ is both $\d_{\pi}$-overconvergent and $\d_{\pi_0}$-overconvergent then
we will later check (cf.  Remark \ref{bibi}) that 
we have a commutative diagram
\begin{equation}
\label{pantof2222}
\begin{tikzpicture}[baseline=(current bounding box.center)]

\node (LT) at (-1,0) {$ V(R_{\pi_0}) $};
\node (RT) at (2,0){$ K_{\pi_0} $};
\node (LB) at (-1,-1){$ V(R_\pi) $};
\node (RB) at (2,-1){$ K_{\pi} $};

\draw[->] (LT) edge node[above]{$p^{-*}f_{\pi_0,*}$}(RT);
\draw[->] (LT) edge  (LB);
\draw[->] (RT) edge (RB);
\draw[->] (LB) edge node[above]{$p^{-*}f_{\pi,*}$} (RB);

\end{tikzpicture}
\end{equation}

\end{remark}

\begin{definition} Let
 $f\in \cO^r(V)$, $f:V(R)\ra R$, be a $\d$-function.
 
 1) $f$ is {\it totally $\d$-overconvergent} if it is $\d_{\pi}$-overconvergent for all $\pi\in \Pi$. 
 
 2) Assume $f$ is totally $\d$-overconvergent and let $\lambda:[1,\infty)\ra [0,\infty)$ be a  real function. We say $f$ has  {\it polar order bounded by} 
$\lambda(x)$ if for any $\pi\in \Pi$,
$f$ is $\d_{\pi}$-overconvergent with polar order at most 
$\lambda(e(\pi))$.

3) Assume $f$ is totally $\d$-overconvergent. We say $f$ is {\it tempered} if 
there exist positive real constants $\kappa_1,\kappa_2$, depending on $f$, such that $f$ has polar order bounded
by the function 
$$\lambda(x)=\kappa_1 \log x +\kappa_2,$$
where $\log$ is the natural logarithm.
\end{definition}

\begin{remark} The sum and product of any two totally $\d$-overconvergent $\d$-functions with polar order bounded by $\lambda_1(x)$ and $\lambda_2(x)$ is totally $\d$-overconvergent  with polar order bounded by 
$$\max\{\lambda_1(x),\lambda_2(x)\}\ \ \ \text{and}\ \ \ \lambda_1(x)+\lambda_2(x),$$ 
respectively. Also any 
 function $V(R)\ra R$  induced by a morphism of $R$-schemes $V\ra {\mathbb A}^1$ is clearly totally $\d$-overconvergent  with polar order bounded by the function $\lambda(x)=0$.  Clearly if $f$ is totally $\d$-overconvergent  with polar order bounded by  $\lambda(x)$ and  $g:Y(R)\ra V(R)$ is induced by a morphisms of $R$-schemes $Y \ra V$ then $f\circ g$ is totally $\d$-overconvergent with polar order bounded by the same $\lambda(x)$.  Also if $f$ is totally $\d$-overconvergent with polar order bounded by $\lambda(x)$ then the composition $\d f:=\d\circ f$ is totally $\d$-overconvergent with polar order bounded by $\lambda(x)+1$.
 Consider the rings
 $$\cO^r(V)^{\dagger \dagger}\subset \cO^r(V)^{\dagger}\subset \cO^r(V)$$
 where $\cO^r(V)^{\dagger}$ is the subring of all totally $\d$-overconvergent $\d$-functions
in $\cO^r(V)$ and $\cO^r(V)^{\dagger\dagger}$ is the subring of all tempered 
functions in $\cO^r(V)^{\dagger}$.
 By the above we have
$$\d\cO^r(V)^{\dagger}\subset \cO^{r+1}(V)^{\dagger},\ \ \ \d\cO^r(V)^{\dagger\dagger}\subset \cO^{r+1}(V)^{\dagger\dagger}.$$ 
In particular, the ring $\cO^r(V)^{\dagger\dagger}$   contains all the elements of the form $\d^i f$ with $i \leq r$ and $f \in \cO(V)$; so it follows from \ref{ident1} below that if $V$ is affine then $\cO^r(V)^{\dagger\dagger}$, hence also $\cO^r(V)^{\dagger}$,  is $p$-adically dense in $\cO^r(V)$. Finally, for any $V$,  the presheaves $\cO_V^{r\dagger}$ and $\cO_V^{r\dagger\dagger}$ on $V$, defined by 
 $$U\mapsto \cO^r(U)^{\dagger},\ \ \ U\mapsto \cO^r(U)^{\dagger\dagger}$$ for $U\subset V$ Zariski open are subsheaves of the sheaf 
 $\cO_V^r$ defined by $U\mapsto \cO^r(U)$. As we will see, if $V/R$ has relative dimension $\geq 1$ and $n\geq 1$ the sheaf inclusion
 \begin{equation}
 \label{notaniso}
 \cO_V^{r\dagger}\ra \cO^r_V
 \end{equation}
 is never an isomorphism; cf.  Remark \ref{bibi}, assertion 3.
\end{remark}

\subsection{$\d^{\text{alg}}$-functions} We now come to the main property of interest.

 \begin{definition}
 Let $V$ be a scheme of finite type over $R$.
 A function $$g:V(R^{\text{alg}})\ra K^{\text{alg}}$$ is a $\d^{\text{alg}}${\it -function}
if for any $\pi \in \Pi$ there exists an integer $\nu\geq 0$ and a $\d_{\pi}$-function
$g_{\pi,\nu}:V(R_{\pi})\ra R_{\pi}$ such that the following diagram is commutative.
\begin{center}
\begin{tikzpicture}

\node (LT) at (-1,0) {$ V(R_{\pi}) $};
\node (RT) at (2,0){$ K_{\pi} $};
\node (LB) at (-1,-1){$ V(R^{\text{alg}}) $};
\node (RB) at (2,-1){$ K^{\text{alg}} $};

\draw[->] (LT) edge node[above]{$p^{-\nu}g_{\pi,\nu}$}(RT);
\draw[->] (LT) edge  (LB);
\draw[->] (RT) edge (RB);
\draw[->] (LB) edge node[above]{$g$} (RB);

\end{tikzpicture}
\end{center}

We say $g$ is {\it tempered} if 
 there exist positive real constants $c_1, c_2$, depending on $g$, such that
$$|g(P)|_p\leq c_1 \cdot e(P)^{c_2},\ \ \ P\in V(R^{\text{alg}})$$
where 
$$e(P):=\min \{e(\pi);\ \pi\in \Pi,\ P\in V(R_{\pi})\}.$$
\end{definition}

 \begin{remark}\label{nicelunch}
 \ 
 
 1) Using  the Equality \ref{caldura} and Remark \ref{sirpac} we see that  any  totally $\d$-overconvergent $\d$-function $$f:V(R)\ra R$$ 
  extends uniquely to 
a $\d^{\text{alg}}$-function
\begin{equation}
\label{finfty}
f^{\text{alg}}:V(R^{\text{alg}})\ra K^{\text{alg}}.\end{equation}
If $f$ is tempered then $f^{\text{alg}}$ is tempered.
Conversely, if a $\d$-function $f:V(R)\ra R$ can be extended to a $\d^{\text{alg}}$-function
then $f$ is totally $\d$-overconvergent. 
By the uniqueness of $f^{\text{alg}}$ above the following properties  hold. These will be used repeatedly later. 

2) For any totally $\d$-overconvergent $\d$-functions $f,g:V(R)\ra R$ and any $\lambda\in R$.
$$(f+g)^{\text{alg}}=f^{\text{alg}}+g^{\text{alg}},\ \ (fg)^{\text{alg}}=f^{\text{alg}}g^{\text{alg}},\ \ \ 
(\lambda \cdot f)^{\text{alg}}=\lambda\cdot f^{\text{alg}},\ \ \ (\phi  f)^{\text{alg}}=\phi f^{\text{alg}}.$$
In particular, if $F$ is a polynomial with $R$-coefficients in a number of variables
and $(f_j)$ is a finite family of $\d$-overconvergent $\d$-functions $V(R)\ra R$ such that
the following polynomial relation between the $\phi^i f_j$'s is satisfied
$$F(...,\phi^i f_j,...)=0\ \ \ \text{on}\ \ \ V(R)$$
 then the same polynomial relation between the $\phi^i f_j^{\text{alg}}$ is satisfied:
$$F(..., \phi^i f_j^{\text{alg}},...)=0\ \ \ \text{on}\ \ \ V(R^{\text{alg}}).$$

3)  For any map $g:X(R)\ra V(R)$ induced by a morphism of $R$-schemes $X\ra V$ 
and for any totally $\d$-overconvergent $\d$-function $f:V(R)\ra R$ we have
$$(f\circ g)^{\text{alg}}=f^{\text{alg}}\circ g^{\text{alg}}:X(R^{\text{alg}})\ra K^{\text{alg}},$$
where $g^{\text{alg}}:X(R^{\text{alg}})\ra V(R^{\text{alg}})$ is the map induced by $X\ra V$ on $R^{\text{alg}}$-points.

4) Let $V=G$ be a group scheme and let $f:G(R)\ra R$ be a  {\it $\d$-character} of $G$. Assume $f$
 is a totally $\d$-overconvergent $\d$-function. We claim that
 $$f^{\text{alg}}:G(R^{\text{alg}})\ra R^{\text{alg}}$$
  is also a homomorphism. 
Indeed,  if 
$$\mu:G\times G\ra G$$
 is the multiplication map the claim above  follows from the uniqueness of $g^{\text{alg}}$ where 
$$g:(G\times G)(R)\ra R,\ \ g(P,Q):=f(\mu(P,Q))-f(P)-f(Q).$$

5) Let $G\times V\ra V$ be an action of a smooth group scheme $G$ over $R$ on a smooth $R$-scheme $V$ and let  $f:V(R)\ra R$ and $\chi:G(R)\ra  R$ be totally $\d$-overconvergent $\d$-functions  such that
$$f(gP)=\chi(g)f(P),\ \ g\in G(R),\ \ P\in V(R).$$
We have 
\begin{equation}
\label{ohgood}
f^{\text{alg}}(gP)=\chi^{\text{alg}}(g)f^{\text{alg}}(P),\ \ g\in G(R^{\text{alg}}),\ \ P\in V(R^{\text{alg}}).\end{equation}

6) Let 
$g_1, g_2:X(R)\ra V(R)$
be two maps defined by morphisms $X\ra V$ of $R$-schemes.
Let $\lambda\in R$ and let
$f:V(R)\ra R$
be a totally $\d$-overconvergent $\d$-function such that 
\begin{equation}
\label{of}
f\circ g_1=\lambda\cdot f\circ g_2:X(R)\ra R.\end{equation}
 Then 
\begin{equation}
\label{ofof}
f^{\text{alg}}\circ g_1^{\text{alg}}=\lambda\cdot f^{\text{alg}}\circ g_2^{\text{alg}}:X(R^{\text{alg}})\ra K^{\text{alg}}.\end{equation}
\end{remark}

\begin{remark}  Throughout this paper we fix once and for all, a prime $p$ and the automorphism $\phi$ of $K^{\text{alg}}$; all other objects are determined by these.
 The choice of $\phi$ on $K^{\text{alg}}$ is of course non-canonical; one can make it ``more" canonical by insisting that $\phi$ is the identity on the maximal totally ramified extension ${\mathbb Q}_p^{\text{tot}}$ of ${\mathbb Q}_p$ in ${\mathbb Q}_p^{\text{alg}}$.
   On a similar note one can ask about the dependence of 
our concepts on the choice of $\phi$ on $K^{\text{alg}}$.
Let $\phi^{(1)},\phi^{(2)}$ be two automorphisms of $K^{\text{alg}}$ extending the automorphism $\phi$ of $K$ and for $\pi\in \Pi$ denote by 
$$\d_{\pi}^{(1)},\d_{\pi}^{(2)}:R_{\pi}\ra R_{\pi}$$
the corresponding $\pi$-derivations. Consider the $K$-automorphism $\sigma$ of $K^{\text{alg}}$ defined by
$$\sigma:=\phi^{(2)}\circ (\phi^{(1)})^{-1}:K^{\text{alg}}\ra K^{\text{alg}}$$
and define the operator
$$\d^{(12)}_{\pi}:R_{\pi}\ra R_{\pi},\ \ \d^{(12)}_{\pi} a:=\frac{\sigma a -a}{\pi},\ \ a\in R_{\pi}.$$
Clearly $\d^{(12)}_{\pi}$ is additive and satisfies the identity
$$\d_{\pi}^{(12)}(ab)=a(\d_{\pi}^{(12)}b)+b(\d_{\pi}^{(12)}a)+\pi(\d_{\pi}^{(12)}a)(\d_{\pi}^{(12)}b),\ \ \ a,b\in R_{\pi}.$$
This is an example of what in \cite{analogues} was called a $\pi$-{\it difference operator}.
We have the following formula:
$$\d_{\pi}^{(2)} a=\frac{\sigma \pi}{\pi} \d_{\pi}^{(1)} a + (\sigma \pi) \d_{\pi}^{(12)}\d_{\pi}^{(1)} a+
\d_{\pi}^{(12)} (a^p),\ \ \ a\in R_{\pi}.$$ 
In particular, any $\d_{\pi}^{(2)}$-function $f_{\pi}:V(R_{\pi})\ra R_{\pi}$ of order $r$ on a smooth affine scheme $V\subset {\mathbb A}^d$ can be represented, in affine coordinates, as
$$f_{\pi}(a)=G(...,(\d_{\pi}^{(12)})^{\alpha_1}(\d_{\pi}^{(1)})^{\beta_1}...
(\d_{\pi}^{(12)})^{\alpha_s}(\d_{\pi}^{(1)})^{\beta_s},...),\ \ \ a\in V(R_{\pi})\subset R_{\pi}^d,$$
where $G$ is a restricted power series with coefficients in $R_{\pi}$ and
$$\sum_{i=1}^s(\alpha_i+\beta_i)\leq 2r,\ \ \ \alpha_i,\beta_i\geq 0.$$
Morally, if one ``adds" $\pi$-difference operators  to the picture then 
the theories for $\phi^{(1)}$ and $\phi^{(2)}$ are related in a simple way.
This can be, of course, formalized but we will not pursue this formalism here.
\end{remark}

\subsection{$\d^{{\mathbb C}_p}$-functions}

\begin{definition}
A function 
$h:V(\OCp)\ra {\mathbb C}_p$
is called a $\d^{{\mathbb C}_p}$-{\it function} if it is continuous in the $p$-adic topologies
and there is a (necessarily unique)  $\d^{\text{alg}}$-function $g$ making the following diagram commute:
\begin{center}
\begin{tikzpicture}

\node (LT) at (-1,0) {$ V(R^{\text{alg}}) $};
\node (RT) at (2,0){$ K^{\text{alg}} $};
\node (LB) at (-1,-1){$ V(\OCp) $};
\node (RB) at (2,-1){$ {\mathbb C}_p $};

\draw[->] (LT) edge node[above]{$g$}(RT);
\draw[->] (LT) edge  (LB);
\draw[->] (RT) edge (RB);
\draw[->] (LB) edge node[above]{$h$} (RB);

\end{tikzpicture}
\end{center}

\end{definition}
 
It is not clear at this point if any $\d^{\text{alg}}$-function $V(R^{\text{alg}})\ra R^{\text{alg}}$
or even any tempered such function
can be extended to a continuous function and hence to a $\d^{{\mathbb C}_p}$-function  $V(\OCp)\ra {\mathbb C}_p$. 
We will prove that such an extension is possible in a series of important cases.

 \subsection{Aim and structure of the paper}
  The purpose of the present paper is to show that the ``main" $\d$-functions appearing in the theory of \cite{char, difmod, book} are totally $\d$-overconvergent. This improves upon the results in \cite[Thm. 1.2]{over} which showed that these $\d$-functions were $\d_{\pi}$-overconvergent for $\pi$ of ramification index at most $p-1$. This allows one to consider the functions $f^{\text{alg}}$ in \ref{finfty} and, in particular, it allows one to consider sets $(f^{\text{alg}})^{-1}(0)$ of solutions in $K^{\text{alg}}$ to the ``arithmetic differential equations" $f$. 
  We summarize some of the main results of the paper in the  Theorems \ref{maintheorem1},
  \ref{maintheorem2}, \ref{maintheorem3}, \ref{maintheorem4}
   below. The objects 
 involved in the statement of this theorem
    were introduced in \cite{Barcau, siegel, char, difmod,  book} and will be reviewed in the body of the paper. They are the main players in the theory and applications of arithmetic differential equations \cite{book}.
  
  \begin{theorem}\label{maintheorem1} Let $A$ be an abelian scheme over $R$ and let 
  $\psi$ be a $\d$-character of $A$; cf.  Definition \ref{charrrr}. Then $\psi$ is a tempered totally $\d$-overconvergent $\d$-function on $A$.\end{theorem}
  
  \begin{theorem}\label{maintheorem2} Let $X$ be a smooth fine moduli space  of principally polarized abelian schemes of dimension $g$ with level structure, $A\ra X$ be the universal abelian scheme, ${\mathcal E}$ be the direct image on $X$ of the relative cotangent sheaf $\Omega_{A/X}$, and let 
  $B\ra X$ be the principal bundle associated to  ${\mathcal E}$.  Let $f$
  be an  isogeny covariant Siegel $\d$-modular form of size $g$ and some basic weight; cf.  section \ref{cinciunu}. Then $f$  defines a tempered totally $\d$-overconvergent $\d$-function on  $B$. \end{theorem}

   \begin{theorem}\label{maintheorem3}  In the situation of Theorem \ref{maintheorem2} assume  $g=1$ and $X$ is the modular curve $X_1(N)$ minus the cusps. Let $f$ be a  normalized newform
 of weight $2$ over ${\mathbb Q}$ and let $f^{\sharp}$ be the attached
 $\d$-modular form of weight $0$; cf.  Section \ref{sasedoi}. Then $f^{\sharp}$ defines a tempered totally $\d$-overconvergent $\d$-function on $X$.\end{theorem}
  
    \begin{theorem}\label{maintheorem4} In the situation of Theorem \ref{maintheorem2} assume  $g=1$, $X$ is the modular curve $X_1(N)$ minus the cusps,
   and consider the open set  $B_{\text{ord}}$ of $B$ lying over   the ordinary locus $X_{\text{ord}}$ of $X$. Let $f$ be an isogeny covariant $\d$-modular form on $X_{\text{ord}}$; cf.  Section \ref{cincipatru}. Then $f$  is a tempered totally $\d$-overconvergent $\d$-function 
  on  $B_{\text{ord}}$.
  \end{theorem}
  
  Theorems \ref{maintheorem1},
  \ref{maintheorem2}, \ref{maintheorem3}, \ref{maintheorem4} are special cases of
  Theorem \ref{overpsi}, Corollary \ref{allrightthen}, Corollary \ref{onherside}, and Corollary \ref{hereyes}, respectively. 
  
  Now the theorems above allow one to consider the corresponding 
  $\d^{\text{alg}}$-functions attached to our $\d$-functions. We would like to understand the sets of zeros of these  $\d^{\text{alg}}$-functions and the extendability of these functions to $\d^{{\mathbb C}_p}$-functions. 
  
 \begin{remark}
  If $\psi:A(R)\ra R$ is as in  Theorem \ref{maintheorem1}
 then, since $\psi^{\text{alg}}:A(R^{\text{alg}})\ra K^{\text{alg}}$ is a homomorphism we get that $\text{Ker}\ \psi^{\text{alg}}$ contains the division hull in $A(R^{\text{alg}})$ of the group $\bigcap_{n\geq 1} p^nA(R)$;  cf.  Corollary \ref{slsl}
 for more on $\text{Ker}\ \psi^{\text{alg}}$. For  the analogous kernel in the case of ${\mathbb G}_m$  see Proposition
\ref{zagyuck}.\end{remark}

On the other hand, for the zero sets of $\d$-modular functions we have the following. 
 
 \begin{theorem}
  If $f$ is any non-zero function as Theorem \ref{maintheorem2}, with $g=1$,  then  the set of ordinary points mapped to $0$ by $f^{\text{alg}}:B(R^{\text{alg}})\ra K^{\text{alg}}$ contains the ordinary points whose Serre-Tate parameters are roots of unity.
  Conversely, if an ordinary point $P\in B_{\text{ord}}(R^{\text{alg}})$ is such that $f^{\text{alg}}(P)=0$ and if the Serre-Tate parameter of $P$ is algebraic over ${\mathbb Q}_p$ then that Serre-Tate parameter  is a root of unity.
  \end{theorem}
  
   See for example, Proposition \ref{cora}. As for extendability to $\d^{{\mathbb C}_p}$-functions we will prove the following. 
 
 \begin{theorem}
  Let $\psi$ be as in  Theorem \ref{maintheorem1}, let $f^{\sharp}$ be as in Theorem \ref{maintheorem3}, and let  $f$ be as in
 Theorem \ref{maintheorem4}.
  Then the functions $\psi^{\text{alg}}$, $(f^{\sharp})^{\text{alg}}$,  and $f^{\text{alg}}$ extend to $\d^{{\mathbb C}_p}$-functions
 $$\psi^{{\mathbb C}_p}:A(\OCp)\ra {\mathbb C}_p,\ \ \ 
 (f^{\sharp})^{{\mathbb C}_p}:X(\OCp)\ra {\mathbb C}_p,\ \ \ 
 f^{{\mathbb C}_p}:B_{\text{ord}}(\OCp)\ra {\mathbb C}_p,$$
 respectively.\end{theorem}

  See Propositions \ref{lockedup} and \ref{lockedin} and Corollary \ref{onherside}. We expect a similar result for the functions
  in Theorems \ref{maintheorem2}. 
  
  \bigskip

  The plan of the paper is as follows. In section 3 we discuss some simple examples of totally $\d$-overconvergent $\d$-functions. In  section 4 we rephrase total $\d$-overconvergence in formal scheme theoretic terms. In sections 5 and 6 we review the ``main" $\d$-functions in our theorems, cf.   \cite{siegel, char, difmod,  book}, and prove their total $\d$-overconvergence; section 5 is devoted to $\d$-modular forms; section 6 is devoted to $\d$-characters of abelian schemes.

  \section{Some simple examples}

\begin{example}
Trivial examples of  totally $\d$-overconvergent $\d$-functions are given by
$$f:{\mathbb A}^1(R)=R\ra R,\ \ \ f(a)=\d^ra,$$
which has polar order at most $r$ or, more generally, functions of the form
$$f:{\mathbb A}^1(R)=R\ra R,\ \ \ f(a)=P(a,\d a,...,\d^ra),$$
where $P$ is a polynomial with $R$-coefficients. 
For such a function $f$ the function $f^{\text{alg}}:{\mathbb A}^1(R^{\text{alg}})\ra K^{\text{alg}}$
trivially extends to a continuous function, which is therefore a $\d^{{\mathbb C}_p}$-function $f^{{\mathbb C}_p}:{\mathbb A}^1(\OCp)\ra {\mathbb C}_p$.
\end{example}

\begin{example}
Here is another trivial example of totally $\d$-overconvergent $\d$-functions. Let $\bZ[\phi]$ be the ring of polynomials with $\bZ$-coefficients  in the
indeterminate $\phi$.  For $w=\sum c_i \phi^i\in \bZ[\phi]$ set $w(p):=\sum c_ip^i$. Assume $m$ is an integer coprime to $p$ and dividing $w(p)$. Let $V={\mathbb G}_m=\text{Spec}\ R[x,x^{-1}]$ and consider the $\d$-function
$$f:V(R)=R^{\times}\ra R,\ \ \ f(a):=a^{\frac{w}{m}}:=a^{\frac{w(p)}{m}}\cdot \prod
\left(\frac{\phi^i(a)}{a^{p^i}}\right)^{\frac{c_i}{m}},$$
where if $\gamma\in \bZ_p$ 
we set
$$\left(\frac{\phi^i(a)}{a^{p^i}}\right)^{\gamma}:=\sum_{s=0}^{\infty}\left(\begin{array}{c}
\gamma\\ s\end{array}\right) \frac{p^sP_{p,i}(a,\d a,...,\d^i a)^s}{a^{p^is}},$$
where $P_{p,i}$ are the unique polynomials in $i+1$ variables with $\bZ$-coefficients satisfying
$$\phi^i(b)=b^{p^i}+pP_{p,i}(b,\d b,....,\d^i b),\ \ \ b\in R.$$
We claim that $f$ is $\d_{\pi}$-overconvergent with polar order   $0$. Indeed we trivially have
$$\left(\frac{\phi^i(a)}{a^{p^i}}\right)^{\gamma}=\sum_{s=0}^{\infty}\left(\begin{array}{c}
\gamma\\ s\end{array}\right) \frac{\pi^s P_{\pi,i}(a,\d_{\pi} a,...,\d_{\pi}^i a)^s}{a^{p^is}},$$
where $P_{\pi,i}$ are the unique polynomials in $i+1$ variables with $R_{\pi}$-coefficients satisfying
\begin{equation}
\label{universal}
\phi^i(b)=b^{p^i}+\pi P_{\pi,i}(b,\d_{\pi} b,....,\d_{\pi}^i b),\ \ \ b\in R_{\pi}.\end{equation}
In particular, $f$ is totally $\d$-overconvergent with polar order bounded by the function $\lambda(x)=0$. Moreover $f^{\text{alg}}:{\mathbb G}_m(R^{\text{alg}})\ra K^{\text{alg}}$
trivially extends to a continuous function, which is therefore a $\d^{{\mathbb C}_p}$-function
$f^{{\mathbb C}_p}:{\mathbb G}_m(\OCp)\ra {\mathbb C}_p$.

\end{example}

\begin{remark}
For the examples of totally $\d$-overconvergent $\d$-functions that will  appear later
it is useful to record the following facts:
\begin{equation}
\label{estimate1}
v_p\left(\frac{\pi^n}{n}\right)\geq -\frac{\log\ e}{\log\ p},\ \ \ e:=e(\pi);
\end{equation}
\begin{equation}
\label{estimate2}
v_p\left(\frac{\pi^n}{n}\right) \ra \infty\ \ \ \text{as}\ \ \ n\ra \infty.
\end{equation}
Indeed we have
$$
v_p\left(\frac{\pi^n}{n}\right)\geq \frac{n}{e}-\frac{\log\ n}{\log\ p}
$$
which gives \ref{estimate2}  and, since the function
$$\theta(x)=\frac{x}{e}-\frac{\log\ x}{\log\ p},\ \ \ x>0,$$
attains its minimum at 
$x_0=\frac{e}{\log\ p}$  the minimum of $\theta(x)$ equals 
$$\theta(x_0)=-\frac{\log\ e}{\log\ p}+\frac{1+\log\log \ p}{\log\ p}\geq -\frac{\log\ e}{\log\ p}$$
which gives \ref{estimate1}.
\end{remark}

\begin{example} 
Let $V={\mathbb G}_m=\text{Spec}\ R[x,x^{-1}]$ and consider the $\d$-function
\begin{equation}
\label{bazz}
\psi:{\mathbb G}_m(R)=R^{\times}\ra R,\ \ \ \psi(a):=\frac{1}{p}\log \left(\frac{\phi(a)}{a^p}\right):=\frac{1}{p}\sum_{n\geq1}(-1)^n\frac{p^n}{n}\left(\frac{\d a}{a^p}\right)^n,\end{equation}
which we refer to as the {\it basic $\d$-character} of ${\mathbb G}_m$; cf.  \cite{char}.
This is a homomorphism and its kernel  $\text{Ker}\  \psi$ is the group of roots of unity
in $R$, i.e., the group of all roots of unity of degree prime to $p$ in 
 ${\mathbb Q}_p^{\text{alg}}$.
We claim that $\psi$ is  totally $\d$-overconvergent. Indeed for $\pi\in \Pi$
we have
$$\psi(a)=\frac{1}{p}\sum_{n\geq1}(-1)^n\frac{\pi^n}{n}\left(\frac{\d_{\pi}a}{a^p}\right)^n\in K_{\pi}.
$$
By \ref{estimate1} and \ref{estimate2} $\psi$ is $\d_{\pi}$-overconvergent with polar order at most
$\frac{\log\ e}{\log\ p}+2$.
 So $\psi$ is totally $\d$-overconvergent with  polar order bounded by the function 
 $$\lambda(x):=\frac{\log\ x}{\log\ p}+2.$$
 In particular, we have an induced homomorphism
 \begin{equation}
 \label{rainu}
 \psi^{\text{alg}}:{\mathbb G}_m(R^{\text{alg}})\ra K^{\text{alg}},\end{equation}
 which trivially extends to a continuous homomorphism ergo is a $\d^{{\mathbb C}_p}$-function
 $$\psi^{{\mathbb C}_p}:{\mathbb G}_m(\OCp)\ra {\mathbb C}_p.$$
 
  \end{example}
 
 \begin{proposition} \label{zagyuck}  For the basic $\d$-character $\psi$ of ${\mathbb G}_m$
 and $P\in {\mathbb G}_m(R^{\text{alg}})=(R^{\text{alg}})^{\times}$ the following hold:.
 1) If $P$ is a root of unity then $\psi^{\text{alg}}(P)=0$.
 
 2) If  $\psi^{\text{alg}}(P)=0$ and $P\in {\mathbb G}_m({\mathbb Q}_p^{\text{alg}})$ then $P$ is a root of unity. \end{proposition}
 
{\it Proof}. Assertion 1 is clear. 
To check assertion 2 assume $P=a$, $\psi^{\text{alg}}(a)=0$. For some $\pi$ we have
$$\frac{\phi(a)}{a^p}\in \text{Ker}(\log:1+\pi R_{\pi}\ra R_{\pi})=\{\zeta\in 1+R_{\pi};\text{ there exists } k \text{ such that } \zeta^{p^k}=1\}.$$
Hence for $b=a^{p^k}$ we have $\phi(b)=b^p$. Since $b$ is algebraic over ${\mathbb Q}_p$ we have $\phi^N(b)=b$ for some $N\geq 1$. So $b=b^{p^N}$ hence $b$ is a root of unity, hence so is $a$.\qed

\begin{example}
\label{notover}
It is easy to give examples of $\d$-functions that are not totally $\d$-overconvergent.
The following is such a function,
$$f:{\mathbb A}^1(R)\ra R,\ \ \ f(a)=\exp (p\d a):=\sum_{n=1}^{\infty}  \frac{p^n}{n!} (\d a)^n, \ \ a\in R.$$
This can be checked directly but also follows
trivially from Remark \ref{bibi}, assertion 1,  and the proof of assertion 2 in Proposition \ref{inj}.
\end{example}

\section{Scheme theoretic view on $\d_{\pi}$-overconvergence}

We review some basic concepts from \cite{char, over}. From now on whenever a $\pi$ is mentioned it will be assumed to be in $\Pi$.

\subsection{$\pi$-jet spaces \cite{char}}\label{ark}
Let  
 $C_{\pi}(X,Y) \in R_{\pi}[X,Y]$ be the polynomial
 \[C_{\pi}(X,Y):=\frac{X^p+Y^p-(X+Y)^p}{\pi}=\frac{p}{\pi}C_p(X,Y).\]
A $\pi$-{\it derivation} from an $R_{\pi}$-algebra
 $A$ into an $A$-algebra $B$ is a map $\d_{\pi}:A \ra B$ such that $\d_{\pi}(1)=0$ and
$$\begin{array}{rcl}
\d_{\pi}(x+y) & = &  \d_{\pi} x + \d_{\pi} y
+C_{\pi}(x,y)\\
\d_{\pi}(xy) & = & x^p \cdot \d_{\pi} y +y^p \cdot \d_{\pi} x
+\pi \cdot \d_{\pi} x \cdot \d_{\pi} y,
\end{array}$$
 for all $x,y \in A$; cf.  \cite{char, Jo}. Here for $x\in A$ we continue to write $x$ in place of $x\cdot 1_B\in B$. Given a
$\pi-$derivation we always denote by $\phi:A \ra B$ the map
$\phi(x)=x^p+\pi \d_{\pi} x$; then $\phi$ is a ring homomorphism. 

A
$\d_{\pi}$-{\it prolongation sequence} is a sequence $S^*=(S^r)_{r \geq 0}$ of  $R_{\pi}$-algebras $S^r$, $r
\geq 0$, together with $R_{\pi}$-algebra homomorphisms, all of which will be denoted by $\varphi:S^r \ra
S^{r+1}$ and $\pi$-derivations $\d_{\pi}:S^r \ra S^{r+1}$ such that
$\d_{\pi} \circ \varphi=\varphi \circ \d_{\pi}$ on $S^r$ for all $r$.  A morphism
of $\d_{\pi}$-prolongation sequences, $u^*:S^* \ra \tilde{S}^*$ is a sequence
$u^r:S^r \ra \tilde{S}^r$ of $R_{\pi}$-algebra
 homomorphisms such that $\delta_{\pi}
\circ u^r=u^{r+1} \circ \d_{\pi}$ and $\varphi \circ u^r=u^{r+1} \circ
\varphi$.  For $w=\sum_{i=0}^r c_i \phi^i\in \bZ[\phi]$  (respectively for $w$ with $c_i \geq 0$), $S^*$ a $\d_{\pi}$-prolongation
sequence, and $x \in (S^0)^{\times}$ (respectively $x \in S^0$) we
can consider the element $x^w:=\prod_{i=0}^r \varphi^{r-i}
\phi^i(x)^{c_i} \in (S^r)^{\times}$ (respectively $x^w \in S^r$).

As in the introduction we may consider the Fermat quotient operator
 $$\d_{\pi}:R_{\pi} \ra R_{\pi},\ \ \ \d_{\pi} x=\frac{\phi(x)-x^p}{\pi};$$ 
 clearly $\d_{\pi}$ is a $\pi$-derivation.
 One can consider the  $\d_{\pi}$-prolongation sequence $R_{\pi}^*$ where
$R_{\pi}^r=R_{\pi}$ for all $r$. By a $\d_{\pi}$-{\it prolongation sequence over $R_{\pi}$} we
understand a prolongation sequence $S^*$ equipped with a morphism
$R_{\pi}^* \ra S^*$. From now on
all our $\d_{\pi}$-prolongation sequences are
assumed to be over $R_{\pi}$.

Assume $\pi|\pi_0$.
Note that if $S^*=(S^r)_{r \geq 0}$ is a $\d_{\pi_0}$-prolongation sequence
 such that each $S^r$ is flat over $R_{\pi_0}$ then the sequence
$$S^*\otimes_{R_{\pi_0}} R_{\pi}=(S^r \otimes_{R_{\pi_0}} R_{\pi})_{r \geq 0}$$
has  a natural structure of $\d_{\pi}$-prolongation sequence.
Indeed, letting $\phi:S^r \ra S^{r+1}$ denote, as usual,
the ring homomorphisms $\phi(x)=x^p+\pi_0\d_{\pi_0} x$ one can extend these $\phi$'s to ring homomorphisms $\phi:S^r \otimes_{R_{\pi_0}}R_{\pi} \ra S^{r+1}\otimes_{R_{\pi_0}}R_{\pi}$
by the formula $\phi(x \otimes y)=\phi(x) \otimes \phi(y)$ where $\phi:R_{\pi}\ra R_{\pi}$ is given, as usual, by $\phi(y)=y^p+\pi \d_{\pi} y$. Then one can define $\pi$-derivations 
$\d_{\pi}:S^r \otimes_{R_{\pi_0}}R_{\pi}\ra S^{r+1}\otimes_{R_{\pi_0}}R_{\pi}$
by 
$$\d_{\pi}(z)=\frac{\phi(z)-z^p}{\pi},\ \ \ z \in S^r\otimes_{R}R_{\pi}.$$
 With these $\d_{\pi}$s the sequence
$S^*\otimes_{R_{\pi_0}}R_{\pi}$ is a $\d_{\pi}$-prolongation sequence.

Denote by ${\bf Prol}_{\pi}$  the category of $\d_{\pi}$-prolongation sequences $S^*=(S^n)$ such that
$S^n$ are Noetherian, $p$-adically complete and flat over $R_{\pi}$.
For any $S^*\in {\bf Prol}_{\pi}$ and any affine $S^0$-scheme of finite type $V=\textrm{Spec } A$ there exists
a  $\d_{\pi}$-prolongation sequence, $A^*=(A^r)_{r \geq 0}$ in ${\bf Prol}_{\pi}$, unique up to a canonical isomorphism, with $A^0=\widehat{A}$, equipped with a morphism $S^*\ra A^*$, such that for any $\d_{\pi}$-prolongation sequence $B^*$ in ${\bf Prol}_{\pi}$ equipped with a morphism $S^*\ra B^*$ 
 and any $S^0$-algebra homomorphism $u:A \ra B^0$ there exists a unique morphism  $u^*:A^* \ra B^*$ over $S^*$ with $u^0=u$.
We define the $\pi$-{\it jet spaces}  of $V$ over $S^*$ as the $p$-adic formal schemes
$J^r_{\pi}(V/S^*):=\text{Spf}\ \widehat{A^r}$. This construction immediately globalizes to the case $V$ is not necessarily affine (such that the construction commutes, in the obvious sense, with \'{e}tale maps). 
For $S^*=R^*_{\pi}$ we write $J^r(V):=J^r(V/R^*_{\pi})$. 
For $V$  smooth over $R_{\pi}$
we have a natural identification
\begin{equation}
\label{ident1}
\cO^r_{\pi}(V)\simeq \cO(J^r_{\pi}(V))
\end{equation}
between the ring $\cO^r_{\pi}(V)$ of $\d_{\pi}$-functions $V(R_{\pi}) \ra R_{\pi}$ and the
ring of global functions $\cO(J^r_{\pi}(V))$; cf.  \cite[Prop. 1.4, Rmk. 1.6]{char}.
 If $V=G$ is a group scheme over $R_{\pi}$ then 
 $$f:J^r_{\pi}(G)\ra \widehat{{\mathbb A}^1_{R_{\pi}}}=\widehat{{\mathbb G}_{a/R_{\pi}}}$$
 is a group homomorphism into the additive group if and only if the corresponding map $G(R_{\pi})\ra R_{\pi}$ is a group homomorphism; recall that such an $f$ is called a $\d_{\pi}$-{\it character} of $G$.
 
 For $\pi=p$ we write $\d:=\d_p$, $J^r(V):=J^r_p(V)$, $\cO^r(V):=\cO^r_p(V)$.

As a prototypical example if $V={\mathbb A}^N_{R_{\pi}}=\textrm{Spec }R_{\pi}[x]$ is the affine space  
where $x$ denotes an $N$-tuple of variables, then $J^r_{\pi}(V)=\textrm{Spf }R_{\pi}[x,\d_{\pi} x,...,\d_{\pi}^r x]^{\widehat{\ }}$ where each $\d_{\pi}x,...,\d_{\pi}^r x$ are new $N$-tuples of variables. More generally, if $V$ is affine and has \'{e}tale coordinates $T:V\ra {\mathbb A}^d_{R_{\pi}}$ then we have natural isomorphisms
$$\cO(J^r_{\pi}(V))\simeq \cO(\widehat{V})[\d_{\pi}T,...,\d_{\pi}^rT]^{\widehat{\ }}.$$

Let $V$ be a scheme over $R$. By  a {\it Frobenius lift} on $X$ we mean a morphism $\phi^V:V\ra V$
 of $\bZ$-schemes, compatible with the automorphism  $\phi$ of $R$, such that $\phi^V$ mod $p$ is the $p$-power Frobenius on $V\otimes_R R/pR$. For $V$ smooth projective over $R$ the existence of a Frobenius lift is equivalent to the existence of a section of $J^1(V)\ra \widehat{V}$. More generally let  $V_{\pi}$ be a scheme over $R_{\pi}$. By  a {\it relative Frobenius lift} on $V_{\pi}$ we mean a morphism $\phi^V:V\ra V$
 of $\bZ$-schemes, compatible with the automorphism $\phi$ of $R_{\pi}$, such that $\phi^V$ mod $\pi$ is the $p$-power Frobenius on $V\otimes_{R_{\pi}} R_{\pi}/\pi R_{\pi}$. For $V_{\pi}$ smooth projective over $R_{\pi}$ the existence of a relative Frobenius lift is equivalent to the existence of a section of $J^1_{\pi}(V_{\pi})\ra \widehat{V_{\pi}}$.

 For any scheme $V/R$ we write 
 $V_{R_{\pi}}:=V \otimes_{R} R_{\pi}$.
  Let $V/R$ be a smooth affine scheme. The $\d$-prolongation sequence
 $(\cO(J^r(V)))_{r \geq 0}$
induces a structure of $\d_{\pi}$-prolongation sequence on
the sequence  $(\cO(J^r(V)) \otimes_{R} R_{\pi})_{n \geq 0}$.
By the universality property of the $\d_{\pi}$-prolongation sequence
 $(\cO(J^r_{\pi}(V_{R_{\pi}})))_{r \geq 0}$ we get, for any $\pi_0$ with $\pi|\pi_0$,  a canonical morphism of $\d_{\pi}$-prolongation sequences
\begin{equation}
\label{tra}\cO(J^r_{\pi}(V_{R_{\pi}}))\ra \cO(J^r_{\pi_0}(V_{\pi_0}))\otimes_{R_{\pi_0}} R_{\pi}.\end{equation}
Note that  the map
\begin{equation}
\label{ris}
\cO(J^r_{\pi}(V_{R_{\pi}}))\ra \cO(J^r(V))\otimes_{R} R_{\pi}\end{equation}
factors through the map \ref{tra}. On the other hand the map \ref{ris} induces a map
\begin{equation}
\label{risp}
\cO(J^r_{\pi}(V_{R_{\pi}}))[1/p]\ra \cO(J^r(V))\otimes_{R} R_{\pi}[1/p].\end{equation}
The maps \ref{tra}, \ref{ris}, \ref{risp} induce such maps for $V$ smooth,  not necessarily affine.

\begin{proposition}
\label{inj} For $V$ smooth,

1) the map (\ref{tra}) is injective.

2) Assume $V/R$ has relative dimension $\geq 1$ and assume  $r\geq 1$ and $e(\pi)$ is sufficiently big (e.g. $\geq p-1$). Then the injective map \ref{risp} is not surjective. 
\end{proposition}

We will usually view the maps \ref{tra} and \ref{ris} as  inclusions. Note that these maps are very far from being flat in general.

\bigskip

{\it Proof}.  Assertion 1 was proved in \cite[Prop. 2.2]{over} but  we need to
 recall the argument. We may assume $V$ is affine and possesses \'{e}tale
coordinates $T:V\ra {\mathbb A}^d$. Then  the source  of \ref{ris} embeds naturally into the ring
$$K_{\pi}[[T,\d_{\pi}T,...,\d_{\pi}^r T]]=K_{\pi}[[T,\phi(T),...,\phi^r (T)]]$$
while the target of \ref{ris} naturally embeds into the same ring,
$$K_{\pi}[[T,\d T,...,\d^r T]]=K_{\pi}[[T,\phi(T),...,\phi^r (T)]].$$
This implies assertion 1. 

To check assertion 2 we may assume again that
$V$ is affine and possesses \'{e}tale
coordinates $T$. Now take any $\pi$ with $e(\pi)=:e\geq p-1$ and consider the element,
$$f=\exp(p\d T):=\sum_{n=1}^{\infty}  \frac{p^n}{n!} (\d T)^n\in \cO(\widehat{V})[\d T,...,\d^r T]^{\widehat{\ }}=\cO(J^r(V)).$$
 Then $f$ identifies in 
the ring 
\begin{equation}
\label{thering}
K_{\pi}[[T,\d_{\pi}T,...,\d_{\pi}^r T]]\end{equation}
 with the element
\begin{equation}
\label{dormi}
\sum_{n=1}^{\infty}  \frac{\pi^n}{n!} (\d_{\pi} T)^n.\end{equation}
Now if $f$ is in the image of the map \ref{risp} then \ref{dormi} would be identified in 
\ref{thering} with an element of the form
\begin{equation}
\label{alt}
\sum_{n=1}^{\infty}  b_n (\d_{\pi} T)^n,
\end{equation}
where $b_n\in K_{\pi}$ satisfy $v_p(b_n)\ra \infty$. Comparing \ref{alt} and \ref{dormi}
we get that 
$$\frac{n}{e}-v_p(n!)\ra \infty\ \ \ \text{as}\ \ \ n\ra \infty$$
which is, of course, false if $e\geq p-1$. This ends the proof of assertion 2.
\qed

\begin{remark}\label{bibi}
Let $V$ be a smooth scheme over $R$.

1) Under the identification 
$\cO^r(V)=\cO(J^r(V))$ (cf.  
\ref{ident1}) a function $f\in \cO^r(V)$ 
is $\d_{\pi}$-overconvergent  with polar order at most $\nu$ if and only if the function
$p^{\nu}f\otimes 1\in \cO(J^r(V))\otimes_{R}R_{\pi}$ belongs to the image of
the map \ref{ris}.

2) 
Proposition \ref{inj}  trivially implies the uniqueness of $f_{\pi,\nu}$ in diagram \ref{pantof2} and the commutativity 
of diagram \ref{pantof2222}.

3) Proposition \ref{inj} also trivially implies that if $V$ is affine, possesses \'{e}tale coordinates, and has relative dimension at least $1$ over $R$ then, for any $r\geq 1$ there exist $\d$-functions on $V$ which are not totally $\d$-overconvergent. This implies the assertion that the inclusion
of sheaves \ref{notaniso} is not an isomorphism.
 \end{remark}

\begin{remark}
Let us say that $f\in \cO^r(V)=\cO(J^r(V))$ has {\it polar $\pi$-order} $N$ if $N$ is the smallest non-negative integer such that the function $f\otimes \pi^N\in \cO(J^r(V))\otimes_R R_{\pi}$
belongs to the image of the map \ref{ris}. It is easy to check that if $f$ has polar order $\nu$ and polar $\pi$-order $N$ then
$$[N/e]\leq \nu\leq [N/e]+1,\ \ \ e:=e(\pi),$$
where $[\ \ ]$ stands for the integral part function. Note that if $f$ has polar $\pi$-order $N$ then the function
$f^*_{\pi}\in \cO_{\pi}^r(V_{R_{\pi}})=\cO(J_{\pi}^r(V_{R_{\pi}}))$ mapping to $f\otimes \pi^N$ via \ref{ris}
satisfies
$$f^*_{\pi} \not\equiv 0 \bmod \pi \textrm{ in } \cO^1(V_{R_{\pi}}).$$
so one can ask for the reduction of $f^*_{\pi}$ mod $\pi$. It is interesting to remark that 
 the reduction of $f^*_{\pi}$ mod $\pi$ is, in general, unrelated to the reduction of $f$ mod $p$.
 Here is a simple example that actually effectively appears in the theory, cf.  for instance, \cite[Eq. 4.43]{local}. Indeed, let $V$ be an affine curve over $R$ possessing an \'{e}tale coordinate $T\in \cO(V)$
 and let $f\in \cO^1(V)=\cO(\widehat{V})[\d T]^{\widehat{\ }}$ have the form
 $$f=\alpha \d T +\beta+p\gamma,\ \ \ \alpha,\beta\in \cO(\widehat{V}),\ \ \ 
 \gamma \in \cO^1(V)$$
 hence
 $$f \equiv \alpha \d T + \beta \bmod p \textrm{ in } \cO^1(V).$$
 However a simple computation shows that if $f$ has polar order $\geq 2$ then
 $$f^*_{\pi} \equiv \gamma^* \bmod \pi \textrm{ in } \cO^1_{\pi}(V_{R_{\pi}}),$$
 where 
 $$\gamma^*:=p\pi^N\gamma\in \cO^1_{\pi}(V_{R_{\pi}}).$$
 In concrete examples one sometimes has information about $\alpha,\beta$ but not about $\gamma$. Hence one has information about $f$ mod $p$ but not about $f^*_{\pi}$ mod $\pi$.
\end{remark}

\subsection{Analytic continuation for $\d^{\text{alg}}$-functions}
Let $V$ be a smooth scheme over $R$ of relative dimension $d$ and let ${\mathfrak R}$ be any of the rings
$R,R_{\pi},R^{\text{alg}},\OCp$. Denote by ${\mathfrak M}$ the maximal ideal of ${\mathfrak R}$. Let $\overline{P}\in V(k)$ be a $k$-point of $V$.
By the {\it unit ball} in $V({\mathfrak R})$ with center $\overline{P}$ we mean the set
 ${\mathbb B}(V,\overline{P},{\mathfrak R})$ of all points $P\in V({\mathfrak R})$ whose image in $V(k)$ equals $\overline{P}$. In particular, $V({\mathfrak R})$ is a disjoint union of unit balls:
$$V({\mathfrak R})=\coprod_{\overline{P}\in V(k)} {\mathbb B}(V,\overline{P},{\mathfrak R}).$$
Fix $\overline{P}\in V(k)$ and fix an isomorphism $\sigma$ between the completion of $V$ along $\overline{P}$ and $\text{Spf}\ R[[t]]$, where $t$ is a $d$-tuple of variables. The the unit ball
${\mathbb B}(V,\overline{P},{\mathfrak R})$ is in bijection with ${\mathfrak M}^d$: the bijection is given by attaching to any $\beta\in {\mathfrak M}^d$ the ${\mathfrak R}$-point  $P\in V({\mathfrak R})$
defined by the homomorphism 
$$\cO(U)\ra R[[t]]\ra {\mathfrak R},\ \ \ t\mapsto t(P):=\beta,$$
 for $U\subset V$ affine, containing $\overline{P}$. Note that ${\mathfrak M}^d$ has the $p$-adic topology hence the ball ${\mathbb B}(V,\overline{P},{\mathfrak R})$ inherits this $p$-adic topology (which is independent of the isomorphism $\sigma$). 
 
 \begin{proposition}\label{anana}
 Let $g:V(R^{\text{alg}})\ra K^{\text{alg}}$ be a $\d^{\text{alg}}$-function. Assume the reduction mod $p$ of $V$ is Zariski connected and $g$ vanishes on one ball ${\mathbb B}(V,\overline{P},R^{\text{alg}})$. Then 
 $g$ vanishes everywhere. 
 \end{proposition}
 
 {\it Proof}. Since the reduction of $V$ mod $p$ is Zariski connected
 we may assume $V$ is affine and has \'{e}tale coordinates $t$ around $\overline{P}$ over $R$. Consider an isomorphism between the completion of $V$ along $\overline{P}$ and $\text{Spf}\ R[[t]]$. For any $\pi$ we have a homomorphism
 $$\cO(J^r_{\pi}(V_{R_{\pi}}))=\cO(\widehat{V_{R_{\pi}}})[\d_{\pi}t,...,\d^r_{\pi}t]^{\widehat{\ }}
 \ra R_{\pi}[[t]][\d_{\pi}t,...,\d^r_{\pi}t]^{\widehat{\ }}.$$
 This homomorphism is injective because the reduction mod $p$ of $V$ is connected.
 Recall that, by definition, $g$ produces $\d_{\pi}$-functions $g_{\pi,\nu}:V(R_{\pi})\ra R_{\pi}$,
 $g\in \cO(J^r_{\pi}(V_{R_{\pi}}))$ for 
 $\nu$ sufficiently big. Then the image 
 $$G_{\pi,\nu}=G_{\pi,\nu}(t,\d_{\pi}t,...,\d^r_{\pi}t)\in R_{\pi}[[t]][\d_{\pi}t,...,\d^r_{\pi}t]^{\widehat{\ }}$$
  of $g_{\pi,\nu}$ has the property that
  $$G_{\pi,\nu}(\lambda,\d_{\pi}\lambda,...,\d^r_{\pi}\lambda)=0\ \ \ \text{for all}\ \ \ \lambda\in (\pi R_{\pi})^d.$$
  This immediately implies that $G_{\pi,\nu}=0$ and hence we get $g_{\pi,\nu}=0$. Since this is true for all $\pi$ we get $g=0$.
 \qed
 
 \bigskip
 
 We have the following direct consequence of Proposition \ref{anana}
 
 \begin{corollary}
 Let $h:V(\OCp)\ra {\mathbb C}_p$ be a $\d^{{\mathbb C}_p}$-function. Assume the reduction mod $p$ of $V$ is Zariski connected and $h$ vanishes on one ball ${\mathbb B}(V,\overline{P},\OCp)$. Then 
 $h$ vanishes everywhere. 
 \end{corollary}

\subsection{$\pi$-jets  of formal groups}
Our aim here is to improve upon a criterion for $\pi$-jets of formal groups in \cite[Sec. 2.2]{over}.

\

Start with a formal group law in $g$ variables, ${\mathcal F}\in S[[{\bf T}_1,{\bf T}_2]]^g$ (where ${\bf T}_1,{\bf T}_2$ are $g$-tuples of variables) over a Noetherian flat $R$-algebra $S$ and let $S^*$ be a $\d_{\pi}$-prolongation sequence in ${\bf Prol}_p$ with $S^0=\widehat{S}$.   One has a natural
$p$-prolongation sequence
$$(S^r[[{\bf T}_1, {\bf T}_2,\d {\bf T}_1,\d {\bf T}_2,...,\d^r{\bf T}_1,\d^r{\bf T}_2]])_{r \geq 0}$$
(where, for $i=1,2$,  $\d{\bf T}_i,\d^2{\bf T}_i,...$ are  $g$-tuples of new variables). Then the $r+1$-tuple of $g$-tuples
\begin{equation}
\label{aboveseries}
{\mathcal F},\d {\mathcal F},...,\d^r {\mathcal F}\end{equation}
defines a commutative formal group in $(r+1)g$ variables; we view these $(r+1)g$ variables  as $r+1$ $g$-tuples 
$${\bf T},\d {\bf T},...,\d^r {\bf T}$$ of variables. Setting ${\bf T}_1={\bf T}_2=0$ in the series \ref{aboveseries}, and forgetting about the first of them, we obtain $r$  $g$-tuples of series
$$F_1:=\{\d {\mathcal F}\}_{|{\bf T}_1={\bf T}_2=0},...,F_r:=\{\d^r {\mathcal F}\}_{|{\bf T}_1={\bf T}_2=0}.$$
The components of this $r$-tuple belong to $S^r[\d {\bf T},...,\d^r {\bf T}]^{\widehat{\ }}$ and define a group
\begin{equation}
\label{groo}
(\widehat{{\mathbb A}^{rg}_{S^r}},[+])\end{equation}
  in the category of $p$-adic formal schemes over $S^r$.
Now let $l({\bf T})=(l_k({\bf T}))_{k=1}^g$,
$$l_k(T)=\sum_{|\alpha|\geq 1} A_{\alpha k}{\bf T}^{\alpha} \in (S\otimes {\mathbb Q})[[T]]$$
be the logarithm of ${\mathcal F}$, where $\alpha=(\alpha_1,...,\alpha_g)\in \bZ_{\geq 0}^g$, $|\alpha|:=\sum\alpha_i$, ${\bf T}=\{T_1,...,T_g\}$, ${\bf T}^{\alpha}:=T_1^{\alpha_1}...T_g^{\alpha_g}$. Recall from 
\cite[p. 64]{Hazewinkel}, for all $\alpha$ we have:
\begin{equation}
\label{hazehaze}
|\alpha|A_{\alpha k} \in S.\end{equation}
  Define
\begin{equation}
\label{Lpr}
L^r=(L^r_k)_{k=1}^g,\ \ \ L^r_k:=\frac{1}{p}\{\phi^r(l_k({\bf T}))\}_{|{\bf T}=0}\in (S^r\otimes {\mathbb Q})[[\d T,...,\d^r T]].\end{equation}
Then $L^r_k$ actually belong to $S^r[\d T,...,\d^r T]^{\widehat{\ }}$ and define group homomorphisms
$$L^r:(\widehat{{\mathbb A}^{rg}_{S^r}},[+])\ra (\widehat{{\mathbb A}^g_{S^r}},+)=\widehat{{\mathbb G}^g_{a,S^r}}.$$
For all the facts above, in case $g=1$ we refer to \cite[pg. 123--125]{book}; the general case $g\geq 1$ is entirely similar.

Now let $X$ be a smooth affine $R$-scheme and set $S^r=\cO(J^r(X))$ and $$S^r_{\pi}:=\cO(J^r_{\pi}(X_{R_{\pi}})) \subset S^r \otimes_{R} R_{\pi},$$
 cf.  Proposition \ref{inj}.
We have the following result, which was proved in \cite[Prop. 2.19]{over} for $g=1$ under a restrictive hypothesis on $e$ that we here remove.

\begin{proposition}
\label{Lrp}
For $e=e(\pi)$ and $\nu:=[\frac{\log\ e}{\log\ p}]+2$ the series 
 $p^{\nu}L^r_k\otimes 1$ belong to the image of the natural homomorphism
$$S^r_{\pi}[\d_{\pi} {\bf T},...,\d_{\pi}^r {\bf T}]^{\widehat{\ }} \ra S^r[\d {\bf T},...,\d^r {\bf T}]^{\widehat{\ }}\otimes_{R} R_{\pi}.$$
\end{proposition}

{\it Proof}.
Since $\phi^r(T_k) \equiv T_k^{p^r}$ mod $\pi$ in
$R_{\pi}[\d_{\pi} {\bf T},...,\d_{\pi}^r {\bf T}]$
we have $\{\phi^r(T_k)\}_{|T_k=0}\equiv 0$ mod $\pi$ in that ring.
Set $G_{r\pi k}=\frac{1}{\pi} \{\phi^r(T_k)\}_{|T_k=0}$, ${\bf G}_{r\pi}=(G_{r\pi k})_{k=1}^g$. Note that
$$p^{\nu}L_k^r=p^{\nu-1}\sum_{|\alpha| \geq 1} \phi^r(|\alpha|A_{\alpha k}) \frac{\pi^{|\alpha|}}{|\alpha |} {\bf G}_{r\pi}^{\alpha}$$
and we are done by \ref{estimate1}, \ref{estimate2}, and \ref{hazehaze}. \qed

\section{Applications to $\d$-modular forms} 

\subsection{Siegel $\d$-modular forms}\label{cinciunu}
 We recall some concepts from \cite{siegel, book} related to Siegel $\d$-modular forms
 and we introduce a construction that will be needed later.

 For any Noetherian
ring $S$ we let ${\bf M}_g(S)$ denote the set of all triples
$(A,\theta,\omega)$ where $A/S$ is an Abelian scheme
of relative dimension $g$, $\theta:A \ra
\check{A}$ is a principal polarization,
and $\omega=(\omega_1,...,\omega_g)^\top$  is a column vector
whose entries are a basis of the $S-$module of $1-$forms
$H^0(A,\Omega^1_{A/S})$. Here and later the superscript $\top$ denotes the transpose
of a matrix.
Recall that we denoted by ${\bf Prol}_p$  the category of $\d$-prolongation sequences $S^*=(S^n)$ such that
$S^n$ are Noetherian, $p$-adically complete and flat over $R$.
By a {\it Siegel $\d-$modular function}
 of genus $g$, size $g$, and order $r$
we understand a rule, call it $f$, that associates to any
 prolongation sequence  $S^* \in {\bf Prol}_p$
and to any triple $(A,\theta,\omega) \in {\bf M}_g(S^0)$
a $g \times g$ matrix
$$f(A,\theta,\omega,S^*) \in \text{Mat}_g(S^r),$$
  depending on $S^*$ and on the isomorphism class
of $(A,\theta,\omega)$ only, such that the formation of  $f(A,\theta,\omega,S^*)$  is functorial
in $S^*$ in the sense that if $\eta:S^* \ra \tilde{S}^*$
is a morphism of  prolongation sequences in ${\bf Prol}_p$
and $\eta^*$ denotes ``pull back via $\eta$" then
$$f(\pi^*A, \eta^*\theta,\eta^*\omega,\tilde{S}^*)=
\eta(f(A,\theta,\omega,S^*)).$$

In \cite[Sec. 1.3]{siegel} arbitrary sizes (rather than size $g$) were considered.
For now we discuss size $g$ forms; we will later discuss the case of size $1$.

By a {\it basic weight} we mean a pair $(\phi^a,\phi^b)\in \bZ[\phi]\times \bZ[\phi]$, where 
 $a,b$ are two non-negative integers. Each $\phi^a$ and $\phi^b$ may be viewed as $\delta$-homomorphism $\textrm{GL}_g(S^0) \to \textrm{GL}_g(S^a)$ and $\textrm{GL}_g(S^0) \to \textrm{GL}_g(S^b)$ respectively by iterating the ring homomorphism $\phi \colon S^i \to S^{i+1}$ where $\phi(x) = x^p + p \delta x$. 
A  Siegel $\d-$modular function $f$ as above is called
a {\it Siegel $\d-$modular form} of   basic weight $(\phi^a,\phi^b)$
if for all $(A,\theta,\omega) \in {\bf M}_g(S^0)$
and $\lambda \in GL_g(S^0)$ we have
$$ f(A,\theta,\lambda \omega,S^*)= \lambda^{\phi^a} \cdot
f(A,\theta,\omega,S^*) \cdot (\lambda^\top)^{\phi^b}$$ where the interpretation of $\lambda^{\phi^a}$ is to apply the associated $\delta$-homomorphism to $\lambda$ and similarly for $(\lambda^\top)^{\phi^b}$. 

\noindent The collection of all Siegel $\d-$modular forms of genus $g$, size $g$, 
order $r$ and weight $(\phi^a,\phi^b)$ will be denoted by 
$M^r_g(\phi^a,\phi^b)$. In \cite[Sec. 1.3]{siegel} more general weights were considered.
We will come back to such weights later.

 Let 
$$(A_1,\theta_1,\omega_1),  (A_2,\theta_2,\omega_2)
\in {\bf M}_g(S)$$
 and let $u:A_1 \ra A_2$ be an isogeny
(which is not assumed to be compatible
with the forms or the polarizations).
We let 
$$u^\top:=\theta_1^{-1} \circ \check{u} \circ 
\theta_2:
A_2 \ra \check{A}_2 \ra \check{A}_1
 \ra A_1$$
denote its ``transpose'' and we let $[u]$ be the unique 
$g \times g$ matrix with $S-$coefficients such that $u^* \omega_2=[u]
\cdot \omega_1$. Let 
$d(u):=det([u \circ u^\top])$. We have $d(u)=\pm \text{deg}(u)$.
If $g=1$ we have $d(u)=deg(u)$. 
   A Siegel $\d-$modular form $f\in M^r_g(\phi^a,\phi^b)$ 
     will be called {\it isogeny
   covariant}  if 
for any
 prolongation sequence
 $S^* \in {\bf Prol}_p$,
  any $(A_1,\theta_1,\omega_1),  (A_2,\theta_2,\omega_2)
\in {\bf M}_g(S^0)$, and any  isogeny $u:A_1 \ra A_2$,
of degree prime to $p$, 
such that $[u]$ is the identity (i.e. $u^* \omega_2=\omega_1$)
the following holds:
\[f(A_2,\theta_2,\omega_2,S^*)= 
f(A_1, \theta_1, \omega_1,S^*) \cdot
(([u^\top]^\top)^{-1})^{\phi^b}.\]

\bigskip

\noindent We denote by $I^r_g(\phi^a,\phi^b)$ the space
of all isogeny covariant Siegel $\d-$modular
forms in $M^n_g(\phi^a,\phi^b)$.  Recall the following:

\begin{theorem} \label{oneone}
\cite[Thm. 1.7]{siegel}
The $R$-module $I^r_g(\phi^a,\phi^b)$ has rank one if 
$r\geq a$, $r\geq b$, $a\neq b$,
and has rank zero in all other cases.
\end{theorem}

 For $r\geq 1$ let $f^r$ be a basis of the rank one
$R$-module   $I^r_g(\phi^r,\phi^0)$.
 For any non-negative integers $a,b$ with $a>b$ set
   \begin{equation}
   \label{formsab}
   f^{ab}:=\phi^b f^{a-b},\ \ \ f^{ba}:=(f^{ab})^\top.\end{equation}
   So
we have $0\neq f^{ab} \in I^c_{g}(\phi^a,\phi^b)$ whenever $a \neq b$
and $c=\text{max}\{a,b\}$. Hence $f^{ab}$ is a basis of $I^r_{g}(\phi^a,\phi^b)\otimes K$ whenever $a \neq b$ and $r\geq c$.

For the  purpose of our proofs later  we need to explicitly construct a basis $f^r$ for $I^r_g(\phi^r,\phi^0)$.
The construction that follows is a generalization for arbitrary $g\geq 1$  of a basic construction in \cite[Sec. 4, Cons. 4.1]{difmod} for $g=1$ and is different from the (crystalline) construction in \cite[Sec. 4.1]{siegel}. By Theorem \ref{oneone}, the two constructions yield (up to a multiplicative constant in $R$) the same form.
  
  To construct $f^r$ start with
  any $S^*\in {\bf Prol}_p$ and 
  let $(A,\omega,\theta)\in {\bf M}(S^0)$.    Cover $A$
by affine open sets $U_i$. Then the natural projections 
$$J^r(U_i/S^*) \ra \widehat{U_i} \otimes_{S^0} S^r$$
 possess sections
$$s_i:\widehat{U_i}\otimes_{S^0} S^r \ra J^r(U_i/S^*).$$
Let 
$$N^r:=Ker(J^r(A/S^*) \ra \widehat{A}\otimes_{S^0} S^r);$$
 it is a group object in the category of $p$-adic formal schemes over $S^r$. Then the differences $s_i-s_j$ define
morphisms
$$s_i-s_j: \widehat{U_{ij}} \otimes_{S^0} S^r \ra N^r$$
where the difference is taken in the group law of $J^r(A/S^*)/S^r$.
On the other hand $N^r$ identifies with the group $(\widehat{{\mathbb A}^{rg}_{S^r}},[+])$ in
(\ref{groo})
with coordinates given by  $\d {\bf T},...,\d^r {\bf T}$, where  ${\bf T}$ is a $g$-tuple of formal parameters at the origin of $A/S$. Let $L^r$ be the tuple of series in (\ref{Lpr}) attached to the formal group of $A$
with respect to  ${\bf T}$, viewed as a homomorphism $L^r:N^r=(\widehat{{\mathbb A}^{rg}_{S^r}},[+])\ra \widehat{{\mathbb G}^g_{a,S^r}}$.
The compositions
\begin{equation}
\label{discri}
L^r \circ(s_i-s_j):\widehat{U_{ij}} \otimes_{S^0} S^r \ra \widehat{{\mathbb G}^g_{a,S^r}}\end{equation}
define a Cech cocycle of elements
\begin{equation}
\label{vv}
\varphi^r_{ij}\in \cO(\widehat{U_{ij}} \otimes_{S^0} S^r)^g\end{equation}
 and hence a column vector 
 $$\varphi^r=(\varphi^{r,1},...,\varphi^{r,g})^\top$$ of cohomology classes in
$$H^1(\widehat{A}\otimes_{S^0} S^r,\cO)\simeq H^1(A\otimes_{S^0} S^r,\cO).$$
   Now $\theta$ induces canonically an identification  of Lie algebras 
$$H^0(A,T_{A/S^0})\simeq H^1(A,\cO)$$
 of $A$ and $\check{A}$ respectively and hence we have  an induced pairing
$$\langle\ \ ,\ \ \rangle_{\theta}:H^1(A,\cO)\times H^0(A,\Omega)\ra S^0.$$
We denote by the same symbol, $\langle \ ,\  \rangle_{\theta}$, the  extensions of this pairing to the pull backs of the above spaces to $S^r$.
Consider the basis
$$\omega=(\omega_1,...,\omega_g)^\top.$$
 Assume that one can choose  \'{e}tale coordinates over $S^0$, 
$${\bf T}:U\ra {\mathbb A}^g_{S^0},$$
 on an open set $U$ of $A$, containing the zero section, such that 
\begin{equation}
\label{ceai}
\omega \equiv d{\bf T} \bmod {\bf T}. 
\end{equation}
This can be  done locally on $\text{Spec}\ S^0$ in the Zariski topology.
Let us take ${\bf T}$ as formal parameters for $A/S^0$ at the origin.
Then one can consider the $g\times g$ matrix
\begin{equation}
\label{nono}
f^r(A,\theta,\omega,S^*):=\langle \varphi^r,\omega^\top\rangle_{\theta}\in \text{Mat}_g(S^r)\end{equation}
with entries
$$f^r(A,\theta,\omega,S^*)_{\alpha \beta}:=\langle\varphi^{r,\alpha},\omega_{\beta}\rangle_{\theta}\in S^r,\ \ \ \alpha,\beta=1,...,g.$$
This matrix does not depend on the choice of ${\bf T}$ as long as \ref{ceai} holds.
In general, when  \'{e}tale coordinates as above do not necessarily exist,  we make the construction locally in the Zariski topology and glue to get a matrix \ref{nono}.
It is easy to check that the association
$$(A,\theta,\omega,S^*)\mapsto f^r(A,\theta,\omega,S^*)$$
defines an element  
\begin{equation}
\label{titi}
f^r\in I^r_g(\phi^r,\phi^0);\end{equation}
 for a similar verification (for a different form) see
\cite[Sec. 4.1]{siegel}.

\medskip

 We next investigate overconvergence of the above construction. We consider the following two special cases which we treat simultaneously:
 
\begin{lemma}\label{froc} Fix $\pi \in \Pi$. In each of the following cases
\begin{enumerate}
\item[Case 1]  $X$ is a smooth affine $R$-scheme with \'{e}tale coordinates $y:X\ra {\mathbb A}^d_R$.
\item[Case 2] $S^r_{\text{for}}:=R[[t]][\d t,...,\d^r t]^{\widehat{\ }}$ with $t$ a $d$-tupe of variables and 
$$S^r_{\pi,\text{for}}:=R[[t]][\d_{\pi} t,...,\d_{\pi}^r t]^{\widehat{\ }}.$$

\end{enumerate}
the function $f^r$ is $\d_{\pi}$-overconvergent. In particular, in Case 1, $f^r$ is totally $\d$-overconvergent of polar order bounded by $\frac{ \log x }{ \log p} + 2$. 
\end{lemma}

\begin{proof}

We handle both cases simultaneously as the proofs are similar. In what follows, set $z=y$, $S^r:=\cO^r(X)$, 
$S^r_{\pi}:=\cO^r_{\pi}(X)$, in Case 1 and set $z=t$, $S^r=S^r_{\text{for}}$, $S^r_{\pi}=S^r_{\pi,\text{for}}$, in Case 2.

Assume $A$ has \'{e}tale coordinates over $S^0$ on an open set $U$ containing the zero section,  ${\bf T}:U\ra {\mathbb A}_{S^0}^g$, satisfying \ref{ceai}, and  each $U_i$ possesses  \'{e}tale coordinates  ${\bf T}_i:U_i \ra  {\mathbb A}^g_{S^0}$ over $S^0$. (We do not need to assume that $U=U_{i_0}$  and ${\bf T}={\bf T}_{i_0}$ for some $i_0$.) Set $e:=e(\pi)$ and  $\nu=[\frac{\log\ e}{\log\ p}]+2$.
Note that
we have a natural morphism
\begin{equation}
\label{messs}
\widehat{U}_i \widehat{\otimes}_{S^0} S^r \otimes_{R} R_{\pi} \simeq (\widehat{U}_i \otimes_{R} R_{\pi}) \otimes_{S_{\pi}^0} (S^r \otimes_{R} R_{\pi})\ra (\widehat{U}_i \otimes_{R} R_{\pi}) \widehat{\otimes}_{S_{\pi}^0}S_{\pi}^r.\end{equation}
Set $\widehat{U}_{i,R_{\pi}}=\widehat{U}_i \otimes_{R} R_{\pi}$.
We claim that one can find sections $s_i$ and $s_{i,\pi}$ of the canonical projections making the following diagram commute:

\begin{equation}
\label{diaa}
\begin{tikzpicture}[baseline=(current bounding box.center)]

\node (LT) at (-3,1) {$\widehat{U}_i \widehat{\otimes}_{S^0} S^r \otimes_{R} R_{\pi}$};
\node (RT) at (3,1){$J^r(U_i/S^*) \otimes_{R} R_{\pi}$};
\node (LB) at (-3,-1){$\widehat{U}_{i,R_{\pi}} \widehat{\otimes}_{S_{\pi}^0}S_{\pi}^r$};
\node (RB) at (3,-1){$J^r_{\pi}(U_{i,R_{\pi}}/S^*_{\pi})$};

\draw[->] (LT) edge node[above]{$s_i$} (RT);
\draw[->] (LT) edge (LB);
\draw[->] (RT) edge (RB);
\draw[->] (LB) edge node[above]{$s_{i,\pi}$} (RB);

\end{tikzpicture}
\end{equation}

Indeed, consider  the commutative diagram

\begin{equation}
\label{diagg}
\begin{tikzpicture}[baseline=(current bounding box.center)]

\node (LT) at (-3,1) {$\cO(\widehat{U}_{i,R_{\pi}})[\d z,...,\d^r z]^{\widehat{\ }} $};
\node (RT) at (3,1){$\cO(\widehat{U}_{i,R_{\pi}})[\d z,...,\d^r z,\d {\bf T}_i,...,\d^r {\bf T}_i]^{\widehat{\ }}$};
\node (LB) at (-3,-1){$\cO(\widehat{U}_{i,R_{\pi}})[\d_{\pi} z,...,\d_{\pi}^r z]^{\widehat{\ }}$};
\node (RB) at (3,-1){$\cO(\widehat{U}_{i,R_{\pi}})[\d_{\pi} z,...,\d_{\pi}^r z,\d_{\pi} {\bf T}_i,...,\d_{\pi}^r {\bf T}_i]^{\widehat{\ }} $};

\draw[->] (RT) edge (LT);
\draw[->] (LB) edge (LT);
\draw[->] (RB) edge (RT);
\draw[->] (RB) edge (LB);

\end{tikzpicture}
\end{equation}

with horizontal arrows sending 

\begin{equation}
\label{luli}
\d {\bf T}_i,...,\d^r {\bf T}_i\mapsto 0,\ \ \ \d_{\pi} {\bf T}_i,...,\d_{\pi}^r {\bf T}_i\mapsto 0.
\end{equation}
 Then the spaces in the diagram
(\ref{diaa}) are the formal spectra of the rings in the diagram (\ref{diagg}) and we can take the horizontal arrows in the diagram
(\ref{diaa}) to be induced by the horizontal arrows in the diagram (\ref{diagg}). The diagram (\ref{diaa}) plus Proposition
\ref{Lrp}
then induces a commutative
diagram

\begin{equation}
\label{striga}
\begin{tikzpicture}[baseline=(current bounding box.center)]

\node (LT) at (0,1) {$\widehat{U}_{ij} \widehat{\otimes}_{S^0} S^r \otimes_{R} R_{\pi} $};
\node (RT) at (3,1){$N^r\otimes_{R} R_{\pi}$};
\node (RTT) at (6,1){$\widehat{\mathbb G}^g_{a,S^r\otimes_{R} R_{\pi}}$};
\node (LB) at (0,-1){$\widehat{U}_{ij,R_{\pi}} \widehat{\otimes}_{S^0_{\pi}} S_{\pi}^r $};
\node (RB) at (3,-1){$N_{\pi}^r$};
\node (RBB) at (6,-1){$\widehat{\mathbb G}^g_{a,S_{\pi}^r}$};

\draw[->] (LT) edge node[above]{$s_i - s_j$} (RT);
\draw[->] (LT) edge (LB);
\draw[->] (RT) edge (RB);
\draw[->] (RTT) edge (RBB);
\draw[->] (LB) edge node[below]{$s_{i,\pi} - s_{j,\pi}$}(RB);
\draw[->] (RT) edge node[above]{$p^\nu L^r$} (RTT);
\draw[->] (RB) edge node[above]{$L_{\pi,\nu}^r$} (RBB);

\end{tikzpicture}
\end{equation}

where $N_{\pi}^r$ is the kernel of the canonical projection 
$$J^r_{\pi}(A_{R_{\pi}}/S^*_{\pi})\ra \widehat{A}_{R_{\pi}} \widehat{\otimes}_{S_{\pi}^0}S_{\pi}^r$$ and the vertical morphisms are the canonical ones.
The diagram (\ref{striga}) shows that the cocycle $p^{\nu}\varphi^r_{ij}$ in (\ref{vv}) comes from a cocycle of elements
in $\cO(\widehat{U}_{ij,R_{\pi}} \widehat{\otimes}_{S_{\pi}^0}S_{\pi}^r)$. This immediately implies that
\begin{equation}
\label{hotz}
p^{\nu}f^r(A,\theta,\omega,S^*)_{\alpha \beta}\otimes 1=p^{\nu} \langle \varphi^{r,\alpha}, \omega_{\beta} \rangle_{\theta} \otimes 1\in \text{Im}(S^r_{\pi}\ra S^r \otimes_{R} R_{\pi}).
\end{equation}

\end{proof}

Denote by 
$f^r_{\pi,\nu}(A,\theta,\omega,S^*)_{\alpha \beta}\in S^r_{\pi}$ the element whose image is
 $p^{\nu}f^r(A,\theta,\omega,S^*)_{\alpha \beta}$.
Case 1 and Case 2 of the above construction are compatible in the following sense.
Let $X$ be as in Case 1, choose a $k$-point $\overline{P}\in X(k)$, and  choose an isomorphism between $R[[z]]$
and the completion of $X$ along $\overline{P}$; we do not assume any compatibility between $y$ and $t$. 
Now consider any $R_{\pi}$-point $P:\text{Spec}\ R_{\pi}\ra X$ reducing to $\overline{P}$ mod $\pi$ and let 
$$\cO(X)\subset R[[t]]\ra R_{\pi},\ \ \ z\mapsto t(P)\in \pi R_{\pi}$$
be the corresponding homomorphism.
Consider a triple ${\mathcal A}:=(A,\theta,\omega)\in {\bf M}(\cO(X))$ and  let 
$${\mathcal A}_{R[[t]]}:=(A_{R[[t]]},\theta_{R[[t]]},\omega_{R[[t]]})\in {\bf M}(R[[t]]),\ \ \ \ {\mathcal A}_P:=(A_P,\theta_P,\omega_P)\in {\bf M}(R_{\pi})$$
 be the base changed triples.
We claim that the diagram 

\begin{center}
\begin{tikzpicture}

\node (LT) at (0,1) {$\cO^r(X)\otimes_R R_{\pi} $};
\node (RT) at (3,1){$S^r_{\text{for}}\otimes_R R_{\pi}$};
\node (RTT) at (6,1){$R_{\pi}$};
\node (LB) at (0,-1){$\cO^r_{\pi}(X)$};
\node (RB) at (3,-1){$S^r_{\pi,\text{for}}$};
\node (RBB) at (6,-1){$R_{\pi}$};

\draw[->] (LT) edge (RT);
\draw[->] (LB) edge (LT);
\draw[->] (RB) edge (RT);
\draw[->] (RBB)edge (RTT);
\draw[->] (LB) edge (RB);
\draw[->] (RT) edge (RTT);
\draw[->] (RB) edge (RBB);

\end{tikzpicture}
\end{center}

induces a diagram

\begin{equation}
\label{lula}
\begin{tikzpicture}[baseline=(current bounding box.center)]

\node (LT) at (0,1) {$ p^{\nu} f^r({\mathcal A},\cO^*(X))\otimes 1$};
\node (RT) at (4,1){$p^{\nu}f^r({\mathcal A}_{R[[t]]},S^*_{\text{for}})\otimes 1$};
\node (RTT) at (8,1){$p^{\nu} f^r({\mathcal A}_P,R_{\pi}^*) $};
\node (LB) at (0,-1){$f^r_{\pi,\nu}({\mathcal A},\cO^*(X))$};
\node (RB) at (4,-1){$f^r_{\pi,\nu}({\mathcal A}_{R[[t]]},S^*_{\text{for}})$};
\node (RBB) at (8,-1){$ f^r_{\pi,\nu}({\mathcal A}_P,R_{\pi}^*)$};

\draw[->] (LT) edge (RT);
\draw[->] (LB) edge (LT);
\draw[->] (RB) edge (RT);
\draw[->] (RBB)edge (RTT);
\draw[->] (LB) edge (RB);
\draw[->] (RT) edge (RTT);
\draw[->] (RB) edge (RBB);

\end{tikzpicture}
\end{equation}

Indeed, this compatibility holds because the horizontal maps in \ref{diagg} can be chosen to be given, in all cases, by the same recipe \ref{luli}; so the fact that $t$ and $y$ are not assumed to be compatible is irrelevant. 

There is a {\it size one} version of theory which we now briefly recall.
Let $w=\sum_{i=1}^r c_i \phi^i\in \bZ[\phi]$; set $\text{deg}(w)=\sum c_i$ and write $w\geq 0$ if $c_i\geq 0$ for all $i$. By an {\it Siegel $\d$-modular form} of order $r$, size $1$, and weight $w$, we understand a rule $f$ that attaches to any 
$S^*\in {\bf Prol}_p$ and any triple $(A,\theta,\omega)\in {\bf M}(S^0)$ an element 
$f(A,\theta,\omega)\in S^r$, depending only on the isomorphism class of the triple, behaving  functorially in $S^*$,  and such that for any $\lambda\in GL_g(S^0)$ we have
$$f(A,\theta,\lambda \omega,S^*)=\text{det}(\lambda)^{-w}f(A,\theta,\omega,S^*).$$
 By an {\it ordinary Siegel $\d$-modular form} of order $r$, size $1$, and weight $w$, we understand a rule as above which is only defined for triples $(A,\theta,\omega)$ with $A/S^0$ having ordinary reduction mod $p$. So $\d$-modular forms as above define ordinary $\d$-modular forms but not vice versa. Denote by $M_{g,1}^r(w)$ the space of  forms as above and by $M^r_{g,1,\text{ord}}(w)$ the space of ordinary forms.

\subsection{Siegel $\d$-modular forms on moduli spaces}\label{cincidoi}
  Consider an abelian scheme $A\ra X$ of relative dimension $g\geq 1$ over a smooth  $R$-scheme $X$. Consider the 
  locally free sheaf  ${\mathcal E}$ on $X$ defined as the direct image of the relative cotangent bundle $\Omega_{A/X}$ and let $B\ra X$
  be the principal $GL_g$-bundle associated to ${\mathcal E}$ (which we refer to as being {\it attached to $A$}); recall that
  \begin{equation}
  B=\text{Spec}(\text{Sym}(\check{\mathcal E}^{\oplus g}))\backslash Z.
  \end{equation}
  where $\check{\ }$ denotes the dual and $Z$ is the obvious ``determinant hypersurface".  To give an  $R$-point of $B$ is the same as to give an $R$-point $P$ of $X$ together with a basis $\omega_P$ for the $1$-forms  on the abelian scheme $A_P$ over $R$ corresponding to $P$.
   By a {\it Siegel $\d$-modular form on $X$ of size $g$, order $r$,  and weight $(\phi^a,\phi^b)$}
   we will mean an element $f\in \cO^r(B)$ such that for any affine open set $U\subset X$ and basis $\omega$ of ${\mathcal E}$ on $U$ 
   we have
    $$f=((\check{\omega})^{\phi^a})^\top \cdot f^{\omega} \cdot 
 (\check{\omega})^{\phi^b},\ \ \text{for some}\ \ f^{\omega}\in \text{Mat}_g(\cO^r(U))$$
 over $U$, where $\check{\omega}$ is the basis of $\check{\mathcal E}$
 dual to $\omega$, i.e. $\langle \check{\omega},\omega^\top\rangle=1_g$. Here both $\omega$ and $\check{\omega}$ are viewed as column vectors.
 Denote by 
 $$M^r_X(\phi^a,\phi^b)\subset \cO^r(B)$$ the module of Siegel $\d$-modular forms of size $g$, order $r$,  and weight $(\phi^a,\phi^b)$ on $X$. There is a natural homomorphism
 \begin{equation}
 \label{bubu}
 M^r_g(\phi^a,\phi^b)\ra M^r_X(\phi^a,\phi^b),\ \ f\mapsto f(A,\theta)\end{equation}
 where $f(A,\theta)$ is defined by gluing the  functions 
 $$((\check{\omega})^{\phi^a})^\top \cdot   f(A,\theta,\omega,\cO^*(U))  \cdot 
 (\check{\omega})^{\phi^b}$$
 where $\cO^*(U)$ is, of course the $\d$-prolongation sequence $(\cO^n(U))$ and $\omega$ is a local basis of ${\mathcal E}$.
 
 The above construction may be applied, of course, to $X$ any open set, with non-empty special fiber at  $p$, of a smooth fine moduli space of principally polarized abelian schemes with some level structure. In this case the map \ref{bubu} is injective.
 
 There is a {\it size $1$} variant of the above.
 Indeed let us continue to assume we have  an abelian scheme $A\ra X$ over a smooth $R$-scheme $X$. 
 Consider the locally free sheaf ${\mathcal E}$ above and its determinant ${\mathcal L}=\text{det}({\mathcal E})$.
Let
 $C\ra X$
  be the principal $GL_1$-bundle associated to ${\mathcal L}$; so
  \begin{equation}
  C=\text{Spec}(\text{Sym}(\check{\mathcal L}))\backslash X=\text{Spec}\left(\bigoplus_{i\in \bZ}{\mathcal L}^{\otimes i}\right),
  \end{equation}
  where $X$ is embedded as the zero section.
   Let $w=\sum_{i=1}^r c_i\phi^i\in \bZ[\phi]$ and set $\text{deg}(w)=\sum c_i$; write $w\geq 0$ if $c_i\geq 0$ for all $i$.
   By a {\it Siegel $\d$-modular form on $X$ of order $r$, size $1$, and weight $w$} we will understand 
   an element $f\in \cO^r(C)$  such for any affine open set $U\subset X$ and basis $\omega=(\omega_1,...,\omega_g)^\top$ of ${\mathcal E}$ on $U$ 
   we have
   \begin{equation}
   f= f_{\omega} \cdot (\omega_1\wedge ... \wedge \omega_g)^w,\ \ \text{for some}\ \ f_{\omega}\in \cO^r(U).\end{equation}
 Denote by 
 $$M^r_X(w)\subset \cO^r(C)$$
  the module of Siegel $\d$-modular forms of order $r$, size $1$, and weight $w$ on $X$. 
  There is a natural map
  $$M^r_{g,1}(w)\ra M_X^r(w);$$
  and if, in addition, $A/X$ has ordinary reduction mod $p$ then there is a natural map
  $$M^r_{g,1,\text{ord}}(w)\ra M_X^r(w).$$

  For any point $P\in X(R)$ we let $A_P$ be the corresponding  
  abelian scheme over $R$; in this case, for any $f\in M^r_X(w)$ and any basis
  of $1$-forms $\omega_P$ on $A_P$
  we set
  $$f(A_P,\omega_P):=f_{\omega}(P)\in R$$
  where $\omega$ is any local basis of ${\mathcal E}$ which reduces to $\omega_P$ at $P$; the definition does not depend, of course, on the choice of $\omega$. 
    
   Now let $\theta$ be a principal polarization on $A/X$.
  Say that $f\in M^r_X(w)$ (where $w$ has even degree) is  an {\it isogeny covariant Siegel $\d$-modular form on $X$ of order $r$, size $1$ and  weight $w$}
  if for all $P_1,P_2\in X(R)$, any isogeny $u:A_{P_1}\ra A_{P_2}$ of degree prime to $p$, and any bases  $\omega_{P_1},\omega_{P_2}$ for the $1$-forms
  on $A_{P_1},A_{P_2}$ respectively, with $u^*\omega_{P_2}=\omega_{P_1}$, we have
  \begin{equation}
  \label{ic}
  f(A_{P_1},\omega_{P_1})=d(u)^{-\frac{\text{deg}(w)}{2}} f(A_{P_2},\omega_{P_2}).\end{equation}
  Note that in the above we do not assume $u$ compatible with the polarizations
  (or with any level structure that might be present). 
  We denote by $I^r_X(w)\subset M^r_X(w)$ the space of isogeny covariant forms. We have a natural map
  $$I^r_g(\phi^a,\phi^b)\ra I^r_X(-\phi^a-\phi^b).$$

 Recall from \cite{siegel} the following construction of elements in $I^r_X(w)$ for various $w$'s. 
 A sequence of non-negative
 integers $a_1,...,a_{2k}$ is said to form a {\it cycle} if
$ a_i \neq a_{i+1}$ for all $1 \leq i \leq 2k-1$ and
$a_{2k} \neq a_1$. Define
$a_{2k+1}:=a_1$. Recall the forms $f^{ab}$ from \ref{formsab}. 
Consider a cycle as above and let
$$w:=\phi^{a_1}-...-\phi^{a_{2k}},\ \ \ r:=a_1+...+a_{2k}.$$
Consider the  Siegel $\d$-modular forms $ f_{a_1...a_{2k}}\in M^r_{g,1}(w)$,
 called in \cite{siegel} {\it cyclic products}, defined by
   \begin{equation}
   \label{cyclic}
   f_{a_1...a_{2k}}:=tr\{f^{a_1a_2} \cdot (f^{a_3a_2})^* \cdot 
   f^{a_3a_4} \cdot (f^{a_5a_4})^* \cdot \cdot \cdot
   f^{a_{2k-1}a_{2k}} \cdot (f^{a_1 a_{2k}})^*\}.\end{equation}
Here the superscript $*$ denotes ``taking the adjoint of a
   matrix''. Of course
$$f_{a_1 a_2 a_3 a_4 ... a_{2k-1} a_{2k}}=
f_{a_3 a_4 ... a_{2k-1} a_{2k} a_1 a_2}$$
and
$$f_{a_1a_2}= g \cdot \text{det}(f^{a_1a_2}).$$
Then $f_{a_1...a_{2k}}$ define elements
$$f_{a_1...a_{2k}}(A,\theta) \in I^r_X(w)\subset \cO^r(C).$$
We refer to \cite{siegel} for applications of 
cyclic products: they may be used to construct ``$\d$-maps" from moduli spaces of abelian schemes to projective spaces, that  are constant on isogeny classes.

  \subsection{Total $\d$-overconvergence  of isogeny covariant Siegel $\d$-modular forms}
  \label{cincitrei}
  
  Our main result here is the following:
  
  \begin{theorem}\label{gogo}
  Let $f\in I^r_g(\phi^a,\phi^b)$ and let $A$ be an abelian scheme over a smooth $R$-scheme $X$. Let $\omega$ be a basis of $H^0(A,\Omega_{A/X})$ and let $\theta$ be a principal polarization on $A$. Then the entries
  \begin{equation}
  \label{circe1}
  f(A,\theta,\omega,\cO(J^*(X)))_{\alpha \beta} \in \cO^r(X),\end{equation}
   of the matrix $f(A,\theta,\omega,\cO(J^*(X)))$ are totally $\d$-overconvergent $\d$-functions on $X$ with  polar order bounded by the function $\lambda(x)=\frac{\log\ x}{\log\ p}+2$.
  \end{theorem}
  
  \begin{corollary}
  \label{allrightthen}
  Let $f\in I^r_g(\phi^a,\phi^b)$ and let $A$ be an abelian scheme over a smooth $R$-scheme $X$. Let  $B\ra X$ be the principal $GL_g$-bundle attached to $A$ and let $\theta$ be a principal polarization on $A$. Then the function
  \begin{equation}
  \label{circe2}
  f(A,\theta)\in \cO^r(B)\end{equation}
  is totally $\d$-overconvergent on $B$ with  polar order bounded by the function $\lambda(x)=\frac{\log\ x}{\log\ p}+2$.
  \end{corollary}
  
  \begin{corollary}
  Let $A$ be an abelian scheme  over a smooth $R$-scheme $X$.
  Let  $C\ra X$ be the principal $GL_1$-bundle attached to $A$ and let $\theta$ be a principal polarization on $A$. Then for any cycle  $a_1,...,a_{2k}$ the cyclic product
  \begin{equation}
  \label{circe3}
  f_{a_1...a_{2k}}(A,\theta)\in \cO^r(C)\end{equation}
  is a totally $\d$-overconvergent $\d$-function on $C$ with  polar order bounded by the function $\lambda(x)=\frac{\log\ x}{\log\ p}+2$.
  \end{corollary}
  
\bigskip

In particular the functions \ref{circe1}, \ref{circe2}, \ref{circe3} are tempered.

\bigskip

{\it Proof of Theorem \ref{gogo}}. 
 By Theorem \ref{oneone}  it is enough to prove our Theorem \ref{gogo} for 
 the forms $f^r$
  constructed in section \ref{cinciunu} which follows immediately from Lemma~\ref{froc}. \qed. 

\

\

\subsection{Case $g=1$}\label{cincipatru}
Form now on we consider the case $g=1$. In this case we
let $X_1(N)$ be the modular curve over $\bZ[1/N,\zeta_N]$ parameterizing elliptic curves with level $\Gamma_1(N)$ structure where $N\geq  4$, $(N,p)=1$
and let $L$ be the line bundle on $X_1(N)$
such that the spaces of sections $H^0(X_1(N),L^{\otimes \kappa})$ identify with the spaces
$M(R,\kappa,N)$ of classical modular forms over $\bZ[1/N,\zeta_N]$ of  weight $\kappa$ on $\Gamma_1(N)$, cf.  \cite[pg. 450]{Gross} where $L$ was denoted by $\omega$.
Choose an embedding $\bZ[1/N,\zeta_N]\subset R$, let $X_1(N)_R$ be the modular curve over $R$, and we continue to denote by $L$ the corresponding bundle on $X_1(N)_R$.
Let $X\subset X_1(N)_R$ be an affine open subset disjoint from the cusps.
The restriction of $L$ to $X$ (still denoted by $L$) is the direct image  ${\mathcal E}=\text{det}({\mathcal E})={\mathcal L}$  of the relative cotangent bundle on the universal elliptic curve $A\ra X$. Here and later by an {\it elliptic curve} we mean an abelian scheme of relative dimension one. So in the situation here $B=C$. An  element $f\in \cO^r(B)=\cO^r(C)$ is called a {\it $\d$-modular form on $X$ of order $r$ and weight} $w\in \bZ[\phi]$ if 
it is a Siegel $\d$-modular form on $X$ of order $r$, size $1$, and weight $w$.
As before we denote by
$$M^r_X(w)\subset \cO^r(B)$$
 the space of  $\d$-modular forms on $X$ of order $r$, size $1$,  and weight $w$ and we denote by
 $$I^r_X(w)\subset M^r_X(w)$$
 the space of isogeny covariant forms in $M^r_X(w)$.
 Note that  $f(A_P,\omega_P)$ in \ref{ic}  can be interpreted as the value of $f$ at the $R$-point of $B$ representing the pair $(A_P,\omega_P)$.    We have a natural homomorphism
 \begin{equation}
 \label{elel}
 M^r_1(\phi^a,\phi^b)\ra M^r_X(-\phi^a-\phi^b),\ \ \ f\mapsto f(A,\theta)\end{equation}
 which, for $X$ the complement of the cusps, is injective;
here $\theta$ is the canonical polarization. By abuse we will write $f$ in place of $f(A,\theta)$.
The map \ref{elel} induces a map
$$I^r_1(\phi^a,\phi^b)\ra I^r_X(-\phi^a-\phi^b)$$
and, in particular,   $f^r$ in \ref{titi} defines an isogeny covariant form, 
$$f^r\in I^r_X(-1-\phi^r)\subset \cO^r(B).$$
We get: 

 \begin{corollary}
 \label{beau}
 The $\d$-functions 
 $f^r$ on $B$ are totally $\d$-overconvergent  with  polar order bounded by the function $\lambda(x)=\frac{\log\ x}{\log\ p}+2$.
 \end{corollary}
 
 In particular $f^r$ are tempered.
 
\medskip
 
 More isogeny covariant forms will be introduced in the special case when $X$ is disjoint from the supersingular locus  and, as we shall see, in that case, all such forms will  be totally $\d$-overconvergent; cf.  Corollary \ref{hereyes}.
 
 On the other hand we have the following general property; roughly speaking if $f$ is isogeny covariant then $f^{\text{alg}}$ has a similar property.

 \begin{proposition}\label{fanfan}
 Let $f\in I^r_X(w)$ be an isogeny covariant $\d$-modular form of order $r$ and weight $w$,  $f:B(R)\ra R$. Assume $f$ is totally $\d$-overconvergent. Then 
 $$f^{\text{alg}}:B(R^{\text{alg}})\ra K^{\text{alg}}$$
  satisfies the following properties.
  
  1) For any $P\in X(R^{\text{alg}})$ let $A_P$ be the corresponding elliptic curve over $R^{\text{alg}}$, let $\omega_P$ be an invertible $1$-form on $A_P$, and let $\lambda\in (R^{\text{alg}})^{\times}$. Then
  \begin{equation}
  \label{dimi}
  f^{\text{alg}}(A_P,\lambda \omega_P)=\lambda^{-w}\cdot f(A_P,\omega_P).
  \end{equation}

   2) Let $P_1,P_2\in X(R^{\text{alg}})$, let $A_{P_1},A_{P_2}$ be the corresponding elliptic curves over $R^{\text{alg}}$, let $\omega_{P_1},\omega_{P_2}$ be invertible $1$-forms on these elliptic curves, respectively, and let $u:A_{P_1}\ra A_{P_2}$ be an isogeny of degree $d$ prime to $p$, preserving the level structures, 
  such that $u^*\omega_{P_2}=\omega_{P_1}$. Then 
  \begin{equation}
  \label{icalg}
  f^{\text{alg}}(A_{P_1},\omega_{P_1})=d^{-\frac{\text{deg}(w)}{2}} f^{\text{alg}}(A_{P_2},\omega_{P_2}).\end{equation}
  \end{proposition}
  
  \bigskip
  
  In the above $f^{\text{alg}}(A_{P_i},\omega_{P_i})\in K^{\text{alg}}$ is, of course,  the value of the function 
$f^{\text{alg}}:B(R^{\text{alg}})\ra K^{\text{alg}}$ at the $R^{\text{alg}}$-point of $B$ represented by the pair $(A_{P_i},\omega_{P_i})$. Also note that we insist here that $u$ preserve the level structures.

\bigskip

{\it Proof}. Assertion 1 can be interpreted as an equality of the form \ref{ohgood}
where $G:={\mathbb G}_m$ acts naturally on $B$ and $\chi(\lambda)=\lambda^{-w}$.

To prove assertion 2 it is enough to assume $d$ is a prime.
Now \ref{icalg} can be interpreted as an equality of the form \ref{ofof} where $V,W,g_1,g_2,\lambda$ in \ref{ofof} are taken  as follows. We take $V=B$,
we let $W$ be scheme over $R$ whose points are identified with 
tuples $(A_1,A_2,\alpha_1,\alpha_2,\omega_1,\omega_2,u)$ where
$A_1,A_2$ are elliptic curves ``lying over $X$", $\alpha_1,\alpha_2$ are  $\Gamma_1(N)$ level structures, $\omega_1,\omega_2$ are invertible $1$-forms, and 
$u$ is an isogeny of degree $d$, respecting the level structures, and pulling back $\omega_2$ to $\omega_1$.
 We let $g_i:W\ra V$ be induced by
 $$(A_1,A_2,\alpha_1,\alpha_2,\omega_1,\omega_2,u)\mapsto (A_i,\alpha_i, \omega_i).$$
 Finally we take $\lambda:=d^{-\frac{\text{deg}(w)}{2}}$. 
 The scheme $W$ can be easily seen to exist as a $GL_1$-bundle over the preimage of $X$ in the ``Hecke correspondence"  parameterizing triples $(A_1,A_2,\alpha_1,\alpha_2, u)$ as above; cf.  \cite[Eq. 2.88]{book}.
  Then the fact that \ref{icalg} holds follows form \ref{ic} plus the our earlier observation that \ref{ofof} follows from \ref{of}.\qed
  
  \bigskip
  
 An analogue of \ref{icalg} can also be proved, in a similar way, for Siegel isogeny covariant $\d$-modular forms.  
  
 \begin{remark}
 By the proof of Theorem \ref{gogo} the following holds. Let 
 $$P\in X(R_{\pi})\subset X(R^{\text{alg}}),$$
  let $A_P$ be  the corresponding elliptic curve over $R_{\pi}$ and let  $\omega_P$ be an invertible $1$-form on $A_P$. Then:
  $(f^1)^{\text{alg}}(A_P,\omega)=0$ if and only if $A_P$ has a relative Frobenius lift.
 Note that for $\pi=p$ the existence of a relative Frobenius lift is equivalent to the existence of a Frobenius lift, which is equivalent to $A_P$ being the {\it canonical lift} of an ordinary elliptic curve; cf.  \cite[Sec. 8.1.6, Cor. 8.68]{book}. \end{remark}
 
 \begin{remark}
 \label{furtherunderstand} 
 To  further understand the zeros of $(f^r)^{\text{alg}}$ we need to review some facts from \cite{book} about {\it $\d$-Serre-Tate expansions}. Let $X\subset X_1(N)$ be disjoint from the cusps and let $\overline{P}\in X(k)$ be an ordinary  point i.e., a $k$-point for which  the corresponding elliptic curve $A_{\overline{P}}/k$ is ordinary.  Let $b$ and $\check{b}$ be bases of the Tate modules of 
 $A_{\overline{P}}$ and of its dual respectively.  Then $b,\check{b}$ define, via Serre-Tate theory \cite{KatzST}, a canonical isomorphism
 between a power series ring in one variable, $R[[t]]$, and the completion of $X$ along $\overline{P}$; for any $R_{\pi}$-point $P$ of $X$ whose reduction mod $\pi$ is $\overline{P}$ the defining homomorphism 
 $$\cO(X)\subset R[[t]]\ra R_{\pi},\ \ \ z\mapsto t(P)\in \pi R_{\pi},$$
  is such that
 the element 
 $$q(A_P)=q(A_P,b,\check{b}):=1+t(P)\in 1+\pi R_{\pi}$$
  is the {\it Serre-Tate parameter} of $A_P$ (with respect to $b,\check{b}$). Consider 
  the diagram \ref{lula} with $A$ the universal elliptic curve over $X$, ${\mathcal A}=(A,\theta,\omega)$, and $\theta$
  the canonical polarization. 
  Recall that we denoted by 
  $${\mathcal A}_{\text{R[[t]]}}=(A_{R[[t]]},\theta_{R[[t]]},\omega_{R[[t]]})$$
   the base changed triple over $R[[t]]$.
  On the other hand there is a ``canonical" triple, 
  $${\mathcal A}_{\text{for}}=(A_{R[[t]]},\theta_{R[[t]]},\omega_{\text{for}})$$
  where $\omega_{\text{for}}$ is 
  the $1$-form on $A_{R[[t]]}$ whose induced $1$-form on the formal group of $A_{R[[t]]}$
  corresponds via $\check{b}$ to the invariant $1$-form on the multiplicative formal group over $R[[t]]$, cf.  \cite[Eq. 8.67]{book} and the discussion before it. So, in particular,
  ${\mathcal A}_{\text{for}}$ does not depend on $\omega$.
     By \cite[Prop. 8.22, 8.61]{book}, there is an $\epsilon\in \bZ_p^{\times}$ such that 
  $$f^r({\mathcal A}_{\text{for}},S^*_{\text{for}})=\epsilon \Lambda^{r-1} \Psi,$$
  where
  $$\Lambda^{r-1}:=\sum_{j=0}^{r-1} p^j \phi^{r-1-j},$$
  and
  $$  \Psi:=
  \frac{1}{p}\sum_{n\geq 1} (-1)^n\frac{p^n}{n}\left( \frac{\d(1+t)}{(1+t)^p}
  \right)^n\in S^1_{\text{for}}=R[[t]][\d t]^{\widehat{\ }}.$$
  The element $f^r({\mathcal A}_{\text{for}},S^*_{\text{for}})$ was called in \cite{book} the {\it $\d$-Serre-Tate expansion} of $f^r$.
  Write 
  $$\omega_{R[[t]]}=u(t)\cdot \omega_{\text{for}},\ \ \ u(t)\in R[[t]]^{\times}.$$
  In particular
  $$f^r({\mathcal A}_{R[[t]]},S^*_{\text{for}})=u(t)^{1+\phi^r}\cdot f^r({\mathcal A}_{\text{for}},S^*_{\text{for}})=\epsilon \cdot u(t)^{1+\phi^r}\cdot \Lambda^{r-1}\Psi.$$
  By the diagram \ref{lula} we must have
  $$f^r_{\pi,\nu}({\mathcal A}_{R[[t]]},S_{\text{for}}^*)=\epsilon\cdot u(t)^{1+\phi^r}\cdot  \Lambda^{r-1} \Psi_{\pi,\nu},$$
  where
  \begin{equation}
  \label{psipinu}
  \Psi_{\pi,\nu}:=p^{\nu-1}\sum_{n\geq 1} (-1)^n\frac{\pi^n}{n}\left( \frac{\d_{\pi}(1+t)}{(1+t)^p}
  \right)^n\in S_{\pi,\text{for}}^1=R_{\pi}[[t]][\d_{\pi} t]^{\widehat{\ }}.\end{equation}
Using the compatibility in \ref{lula} we get

\medskip

\begin{equation}
\begin{array}{rcl}
(f^r)^{\text{alg}}(A_P,\omega_P) & = & p^{-\nu}f^r_{\pi,\nu}({\mathcal A}_P,R^*_{\pi})\\
\ & \ & \ \\
\ & = & \epsilon  p^{-\nu}  
u(t(P))^{1+\phi^r} \Lambda^{r-1}(\Psi_{\pi,\nu}(q(A_P),\d_{\pi}q(A_P)))\\
\ & \ & \ \\
\ & = & \epsilon  u(t(P))^{1+\phi^r}  \Lambda^{r-1}\left(
\frac{1}{p} \log \left( \frac{\phi(q(A_P))}{q(A_P)^p}\right)\right).
\end{array}\end{equation}

\medskip

\noindent So if we identify 
\begin{equation}
\label{identifyy}
X\times {\mathbb G}_m\simeq B\ \ \ \text{via}\ \ \ (P,\lambda)\mapsto (A_P,\lambda \omega_P),
\end{equation}
then $(f^r)^{\text{alg}}$ is given by 
\begin{equation}
\label{fotoop}
(f^r)^{\text{alg}}(P,\lambda)=\lambda^{1+\phi^r}u(t(P))^{1+\phi^r}  \Lambda^{r-1}\left(
\frac{1}{p} \log \left( \frac{\phi(q(A_P))}{q(A_P)^p}\right)\right)
\end{equation}
for $(P,\lambda)$ in the ball ${\mathbb B}(X\times {\mathbb G}_m,(\overline{P},1),R^{\text{alg}})$.

Since $\Lambda^{r-1}:K^{\text{alg}}\ra K^{\text{alg}}$ is injective we have that 
 $$(f^r)^{\text{alg}}(A_P,\omega_P)=0\ \ \ \Leftrightarrow \ \ \ \log \left( \frac{\phi(q(A_P))}{q(A_P)^p}\right)=0$$
In view of Proposition \ref{zagyuck} we get the following result. 
  \end{remark}

  \begin{proposition}\label{cora}
  Let $r\geq 1$ and let $P\in X(R^{\text{alg}})$ be such that the corresponding elliptic curve $A_P/R$ has ordinary reduction 
  and Serre-Tate parameter $q(A_P)$. The following hold:
  
  1) If $q(A_P)$ is a root of unity then 
  $(f^r)^{\text{alg}}(A_P,\omega_P)=0$.
  
  2) If $(f^r)^{\text{alg}}(A_P,\omega_P)=0$ and $q(A_P)$ is algebraic over ${\mathbb Q}_p$ then $q(A_P)$ is a root of unity.
  \end{proposition}
  
  \begin{remark}   The condition that the Serre-Tate parameter is a root of unity is related to the notion of {\it quasi-canonical lift} in the ordinary case; cf.  \cite[Prop. 3.5.1]{goren}.
   It would be interesting to understand what ``non-ordinary zeros" (if any) the functions $(f^r)^{\text{alg}}$ have. 
  \end{remark}

We consider, in what follows, the special case  when 
$$X=X_{\text{ord}} \subset X_1(N)_{R}$$ is the complement of 
the cusps and the  supersingular locus. Recall from \cite[Sec. 3]{Barcau}
that there exist unique ordinary Siegel $\d$-modular forms 
$$f^{\partial}\in M^1_{1,1,\text{ord}}(\phi-1),\ \ \ f_{\partial}\in M^1_{1,1,\text{ord}}(1-\phi),\ \ \ f^{\partial}f_{\partial}=1,$$ 
whose images in $M^1_{X_{\text{ord}}}(\phi-1)$ and $M^1_{X_{\text{ord}}}(1-\phi)$
we continue to denote by $f^{\partial}$ and $f_{\partial}$, satisfying the following properties.
 First these forms are  isogeny covariant, so
$$f^{\partial} \in I^1_{X_{\text{ord}}}(\phi-1),\ \ f_{\partial}\in I^1_{X_{\text{ord}}}(1-\phi).$$
Moreover  consider the  canonical $R$-derivation
$\partial:\cO(B_{\text{ord}})\ra \cO(B_{\text{ord}})$ defined by Katz \cite{Katz} via the Gauss-Manin connection, generalizing the ``Serre operator"; cf.  \cite[pg. 254--255]{book},  consider the conjugate differential operators $\partial_0, \partial_1:\cO^1(B_{\text{ord}}) \ra \cO^1(B_{\text{ord}})$ introduced in \cite[pg. 93]{book} and consider
 the {\it Ramanujan form} $P \in M^0_{X_{\text{ord}}}(2)$; cf.  \cite[pg. 255]{book}. Then one has the following formulae:
\begin{equation}
\label{deffpartial}
f^{\partial}:=\partial_1 f^1-p P^{\phi}f^1 \in \cO^1(B_{\text{ord}}).
\end{equation}
\begin{equation}
\label{foo}
f_{\partial}=-\partial_0 f^1+P \cdot f^1\in \cO^1(B_{\text{ord}}).
\end{equation}

Recall the notation $f^1$ is from (\ref{titi}) with $r = 1$ and is not a power. We have the following consequence involving overconvergence.

\begin{corollary}
\label{manincc}
The $\d$-functions 
$f^{\partial}$ and $f_{\partial}$ on $B_{\text{ord}}$ are totally $\d$-overconvergent with  polar order
 bounded by the function $\lambda(x)=\frac{\log\ x}{\log\ p}+2$.
\end{corollary}

In particular $f^{\partial}$ and $f_{\partial}$  are tempered.

\medskip

{\it Proof}. The argument is the same as the one in the proof of \cite{over}, Theorem 5.3,
except that throughout that proof $\frac{p}{\pi}$ needs to be replaced by $p^{\nu}$   and instead of \cite{over}, Proposition 2.19, one needs to use  Proposition \ref{Lrp} in the present paper.
Also one needs the fact that if $f\in \cO^1(B_{\text{ord}})$ is any element which is $\d$-overconvergent with polar order at most $\nu$ for some $\nu$ then $\partial_0 f, \partial_1 f\in \cO^1(B_{\text{ord}})$ are also
 $\d$-overconvergent with polar order at most $\nu$; the latter follows from  in \cite{over}, 
 Proposition 2.20.
 \qed

\bigskip

On the other hand the forms $f^1$ and $f^{\partial}$ and their iterated images by $\phi$
generate modulo torsion the space of isogeny covariant $\d$-modular forms. Indeed, recall the following basic result explained in \cite{book}.

\begin{theorem}
\label{starrr}
For any $w\in \bZ[\phi]$ of even degree the $K$-vector space
$$I^r_{X_{\text{ord}}}(w)\otimes_R K$$
is spanned by elements of the form
$$(f^1)^v(f^{\partial})^{v'}$$
where $v,v'\in \bZ[\phi]$, $v\geq 0$, $(-1-\phi)v+(\phi-1)v'=w$.
\end{theorem}

{\it Proof}. By \cite[Prop. 8.75, 8.22]{book} the above holds if, in the definition of isogeny covariance, we insist that isogenies be compatible with level structures. But then the result follows because $f^1$ and $f^{\partial}$ satisfy the definition of isogeny covariance given in the present paper (where compatibility with level structures is not being assumed).
\qed

\bigskip

Putting together  Theorem \ref{starrr} and Corollaries \ref{beau} and \ref{manincc} we get the following. 

\begin{corollary}
\label{hereyes}
Any element of $I^r_{X_{\text{ord}}}(w) \subset \cO^r(B_{\text{ord}})$ is a tempered totally $\d$-overconvergent $\d$-function on $B_{\text{ord}}$. In particular, for any 
$$f\in I^r_{X_{\text{ord}}}(w),\ \ \ f:B_{\text{ord}}(R)\ra R$$
 the associated map 
$$f^{\text{alg}}:B_{\text{ord}}(R^{\text{alg}})\ra K^{\text{alg}}$$
satisfies the conclusions of Proposition \ref{fanfan}.
\end{corollary}

\begin{remark}
Recall from \cite[Cor. 8.62, Prop. 8.76]{book},  that the forms $\phi^i f^r$, $\phi^i f^{\partial}$, $\phi^if_{\partial}$  satisfy certain basic polynomial relations, by Remark \ref{nicelunch} the same polynomial relations will be satisfied by the forms $\phi^i(f^r)^{\text{alg}}$, $\phi^i(f^{\partial})^{\text{alg}}$, $\phi^i(f_{\partial})^{\text{alg}}$.
\end{remark}

\begin{remark}
Let us go back to the notation in Remark \ref{furtherunderstand}. 
By \cite[Prop. 8.59]{book}, we have
$$f^{\partial}({\mathcal A}_{\text{for}},S^*_{\text{for}})=\epsilon,$$
 hence, by the same reasoning as in Remark  \ref{furtherunderstand}, we have
 that if $(A,\theta,\omega)$ is a triple over a Zariski open set of $X$ then
\begin{equation}
\begin{array}{rcl}
(f^{\partial})^{\text{alg}}(A_P,\omega_P) & = & p^{-\nu}f^{\partial}_{\pi,\nu}({\mathcal A}_P,R^*_{\pi})= \epsilon \cdot u(t(P))^{1-\phi}.
\end{array}\end{equation}
So, again, under the identification
\ref{identifyy} we have
\begin{equation}
\label{fotoop2}
(f^{\partial})^{\text{alg}}(P,\lambda)=\epsilon \lambda^{1-\phi}u(t(P))^{1-\phi}
\end{equation}
for $(P,\lambda)$ in the ball ${\mathbb B}(X\times {\mathbb G}_m,(\overline{P},1),R^{\text{alg}})$.
In particular we get the following:
\end{remark}

\begin{proposition}
\label{lockedup}
For any $f\in I^r_{X_{\text{ord}}}(w)$ (with $w$ of even degree) the function
$$f^{\text{alg}}:B_{\text{ord}}(R^{\text{alg}})\ra R^{\text{alg}}$$
extends to a  $\d^{{\mathbb C}_p}$-function
$$f^{{\mathbb C}_p}:B_{\text{ord}}(\OCp)\ra {\mathbb C}_p.$$
\end{proposition}

{\it Proof}.
By Theorem \ref{starrr} it is enough to prove the Proposition  for $f=f^1,f^{\partial},f_{\partial}$. 
The statement only needs to verified inside each unit ball. But then the statement follows from
equations \ref{fotoop} and \ref{fotoop2}.
\qed

\begin{remark}
Consider the set $X_{\text{ord}}(\OCp)^{\circ}$ of all points 
$P\in X_{\text{ord}}(\OCp)$ such that 
$$(f^1)^{{\mathbb C}_p}(A_P,\omega_P)\neq 0.$$
The above condition does not depend on $\omega_P$. Then consider the set theoretic map
\begin{equation}
\wp:X_{\text{ord}}(\OCp)^{\circ}\ra {\mathbb C}_p^{\times} \end{equation}
defined by
\begin{equation}
\label{dpm}
\wp(P):=(((f^1)^{{\mathbb C}_p})^{\phi-1}((f^{\partial})^{{\mathbb C}_p})^{\phi+1}(A_P,\omega_P).\end{equation}
The right hand side of the above equation does not depend on $\omega_P$. We may refer to the map $\wp$ as the $\d$-{\it period map}; its restriction to $R$-points played a key role in \cite{book}.  Note that, in our formalism, it does not make sense to say that $\wp$ is a $\d^{{\mathbb C}_p}$-function because $\wp$ is not defined on the whole of $X_{\text{ord}}(\OCp)$. One can extend our formalism to accommodate functions such as $\wp$ but we will not pursue this here; be that as it may $\wp$ is a ``quotient of $\d^{\text{alg}}$-functions". Note also that, by 
\ref{fotoop} and \ref{fotoop2}, we have:
\begin{equation}
\label{gorgona}
\wp(P)=\epsilon^2 \frac{1}{p}\left( \log \left( \frac{\phi(q(A_P))}{q(A_P)^p}\right)\right)^{\phi-1}
\end{equation}
in the ball ${\mathbb B}(X,\overline{P},\OCp)$.
\end{remark}

\begin{remark}
Note that the right hand sides of \ref{fotoop}, \ref{fotoop2}, and \ref{gorgona} are defined locally on balls
centered at ordinary points. The remarkable fact is that these locally defined functions  arise from globally defined  $\d^{{\mathbb C}_p}$-functions.
\end{remark}

\begin{remark}
Consider  the $\d$-period map $\wp$ in \ref{dpm} and fix $\overline{P}\in X_{\text{ord}}(k)$.
Let  $q=1+t$ be the Serre-Tate parameter. Consider the function
\begin{equation}
\label{aprod}
\sigma:{\mathbb A}^1(\OCp)\ra X_{\text{ord}}(\OCp),\ \ u\mapsto P_u,
\end{equation}
where $P_u$ is the point defined by the homomorphism 
$$R[[t]]\ra \OCp,\ \ \ t\mapsto \exp\left(\sum_{n=1}^{\infty} p^n \phi^{-n}(u)\right)-1.$$
Here $\phi^{-1}$ is the inverse of $\phi$ on ${\mathbb C}_p$
Note that $\sigma(0)=P_0\in X_{\text{ord}}(R)$ is the point corresponding to the elliptic curve over $R$ that is the  canonical lift of the elliptic curve corresponding to $\overline{P}$.
Let 
$$\chi:{\mathbb G}_m(\OCp)\ra {\mathbb G}_m(\OCp),\ \ \ \chi(u):=u^{\phi-1}.$$
Using Equation \ref{gorgona} it
 is trivial to verify that 
$\sigma$ maps ${\mathbb G}_m(\OCp)$ into $X_{\text{ord}}(\OCp)^{\circ}$ and
the following diagram is commutative

\begin{equation}
\label{ssoo}
\begin{tikzpicture}[baseline=(current bounding box.center)]

\node (LT) at (0,1) {${\mathbb A}^1(\OCp) $};
\node (RT) at (3,1){${\mathbb G}_m(\OCp)$};
\node (RTT) at (6,1){${\mathbb G}_m(\OCp)$};
\node (LB) at (0,-1){$X_{\text{ord}}(\OCp)$};
\node (RB) at (3,-1){$X_{\text{ord}}(\OCp)^{\circ}$};
\node (RBB) at (6,-1){${\mathbb C}_p^{\times}$};

\draw[->] (RT) edge  (LT);
\draw[->] (LT) edge node[left]{$\sigma$} (LB);
\draw[->] (RT) edge node[left]{$\sigma$}(RB);
\draw[->] (RTT) edge node[right]{$\epsilon^2$} (RBB);
\draw[->] (RB) edge (LB);
\draw[->] (RT) edge node[above]{$\chi$} (RTT);
\draw[->] (RB) edge node[above]{$\wp$} (RBB);

\end{tikzpicture}
\end{equation}

where the right vertical arrow is the multiplication by $\epsilon^2$.
The map $\sigma$ in \ref{aprod} is far from being an object of the $\d$-geometry developed in \cite{book}:
indeed in its definition the {\it negative} (rather than the positive) powers of $\phi$ are involved.
However the map $\sigma$  can be shown to fit into the {\it perfectoid $\d$-geometry} framework 
developed in \cite{perf}. It is an analogue, for modular curves,  of the co-characters of elliptic curves in. We will not pursue this here.
\end{remark}

 \section{Applications to $\d$-characters}

\subsection{Total $\d$-overconvergence of $\d$-characters}
Let  $A/R$  be an abelian scheme of arbitrary relative dimension $g\geq 1$.
Recall from \cite[pg. 310, Sec. 0.3]{char} that a {\it $\d$-character} of order $r$ on $A$ is a $\d$-function 
$\psi:A(R)\ra R$ of order $r$ which is a group homomorphism (with $R$ viewed as a group with its additive structure). We view such a $\psi$  as an element of $\cO^r(A)=\cO(J^r(A))$. It follows from the theory in  \cite[Cor. 2.10]{char} that the $R$-module of $\d$-characters of order $r$ is finitely generated of rank between $(r-1)g$ and $rg$;  we will give an argument for this below; cf Remark \ref{bel}.
  When in addition, $g=1$ we can be more specific. If $A$ does not have a Frobenius lift then it was proved in \cite[Prop. 3.2]{char} that there exists a  $\d$-character
 \begin{equation}
\label{deltacharacter}
\psi\in \cO^2(A)\end{equation}
which is a basis for the $R$-module of $\d$-characters of order $2$; on the other hand, in case $A$ has a Frobenius lift,
there exists a $\d$-character 
 \begin{equation}
\label{woo}
\psi\in \cO^1(A).\end{equation}
which is a basis for the $R$-module of $\d$-characters of order $1$. In each of the cases we can refer to $\psi$ above as a {\it basic $\d$-character} of $A$.

The following result was proved in \cite[Sec. 5.4]{over} in the special case when $g=1$ and $e\leq p-1$.

\begin{theorem}
\label{overpsi}
Let $A/R$ be an abelian scheme of relative dimension $g\geq 1$ and let $r\geq 1$ be an integer. Then there exists an integer $\kappa$ such that for any $\d$-character $\psi\in \cO^r(A)$ of order $r$  of $A$,
   $\psi$ is totally $\d$-overconvergent  with  polar order bounded by the function $\lambda(x)=\frac{\log\ x}{\log\ p}+\kappa$.
\end{theorem}

In particular, any $\d$-character of an abelian scheme over $R$ is a tempered totally $\d$-overconvergent $\d$-function.

\medskip

{\it Proof}. Fix $r\geq 1$. We need to find $\kappa$ such that for any $\d$-character $\psi$ on $A$ of order $r$ and any 
 $\pi\in \Pi$, if $e:=e(\pi)$ and $\nu:=[\frac{\log\ e}{\log\ p}+\kappa]$ then
$p^{\nu}\psi\otimes 1$
 belongs to the image of the map
$$\cO(J^r_{\pi}(A_{R_{\pi}})) \ra \cO(J^r(A))\otimes_{R}R_{\pi}.$$

As explained in \cite{char} the $R$-module of $\d$-characters of order $r$ identifies with the module $\Hom(J^r(A),\widehat{{\mathbb G}_a})$. Recall also that $N^r = \ker(J^r(A) \rightarrow A)$ is the kernel of canonical projection yielding the short exact sequence $0 \to N^r \to J^r(A) \to A \to 0$. Applying $\Hom(-,\widehat{{\mathbb G}_a})$, we get the standard exact sequence,

 \begin{equation}
 \label{pairofmat}
  0=H^0(A,\cO)
 \ra \Hom(J^r(A),\widehat{{\mathbb G}_a})\stackrel{res}{\ra} \Hom(N^r,\widehat{{\mathbb G}_a})\stackrel{\partial}{\ra} H^1(A,\cO)\simeq R^g.\end{equation}
 By \cite[Cor. 2.10]{char}, the $R$-module $\Hom(N^r,\widehat{{\mathbb G}_a})$ is finitely generated of rank $\leq rg$. On the other hand consider the setting of \ref{discri} where $X=\textrm{Spec } R$ and consider the functions denoted there by
 \begin{equation}
 \label{shoapte}
 L^s_k,\ \ \ s=1,...,r,\ \ \ k=1,...,g.\end{equation}
 These functions define, and hence can be identified with,  $R$-linearly independent elements in 
 \begin{equation}
 \label{fury}
 \Hom(N^r,\widehat{{\mathbb G}_a});\end{equation}
 so \ref{fury} has rank $rg$. In particular there exists an integer $\mu$ such that for any $\psi\in \Hom(J^r(A),\widehat{{\mathbb G}_a})$ we have that $p^{\mu}\psi$ restricted to $N^r$ is an $R$-linear combination $L$ of the 
 functions \ref{shoapte}. The image of $L$ via $\partial$ in $H^1(A,\cO)$ is the class of the cocycle
 $$L\circ (s_i-s_j):\widehat{U_{ij}}\ra \widehat{{\mathbb G}_a^g}.$$
 Since this class vanishes there exist elements $\Gamma_i\in \cO(\widehat{U_i})$ such that
 \begin{equation}
 \label{locsol}
 L\circ (s_i-s_j)=\Gamma_i-\Gamma_j.\end{equation}
 We may assume there is an index $i_0$ such that the origin of $A$ belongs to $U_{i_0}$ and $\Gamma_{i_0}$ vanishes at the origin.
 Viewing $J^r(A)$ as obtained from the schemes $\widehat{U_i}\times \widehat{{\mathbb A}^{rg}}$ via the isomorphisms given on points by
 $$(u,v)\mapsto (u,v+s_i(u)-s_j(u))$$
 we see that the condition \ref{locsol} implies that the functions
 $$\psi'_{i,p}:=L-\Gamma_i\in \cO(\widehat{U_i}\times \widehat{{\mathbb A}^{rg}})$$
 glue together to give a function $\psi'\in \cO(J^r(A))$.  Clearly $\psi'(0)=0$.
  Then, as in \cite{book}, Lemma 
 7.13, we get that $\psi'$ is a homomorphism, i.e., it belongs to
 $\Hom(J^r(A),\widehat{{\mathbb G}_a})$.
 Since the map $res$ in \ref{pairofmat} is injective we get that $p^{\mu}\psi=\psi'$. Set $\kappa:=\mu+2$. Then, by 
 Proposition \ref{Lrp},
 the functions $p^{\nu-\mu} \psi'_{i}$ (which glue to give $p^{\nu}\psi$) 
  belong to the image of
$$\cO(\widehat{U_i}\otimes_{R}R_{\pi})[\d_{\pi} {\bf T},..., \d_{\pi}^r{\bf T}]^{\widehat{\ }} \ra \cO(\widehat{U_i}\otimes_{R}R_{\pi})[\d {\bf T},..., \d^r{\bf T}]^{\widehat{\ }}.$$
In order to conclude we 
note that, using the sections $s_i$ and $s_{i,\pi}$ in the diagrams
\ref{diaa}
one has commutative diagrams

\begin{center}
\begin{tikzpicture}

\node (LT) at (-3,1) {$(\widehat{U_i}\otimes_{R}R_{\pi})\widehat{\times} (N^r\otimes_{R}R_{\pi}) $};
\node (RT) at (3,1){$ J^r(U_i)\otimes_{R}R_{\pi}$};
\node (LB) at (-3,-1){$(\widehat{U_i}\otimes_{R}R_{\pi}) \widehat{\times} N^r_{\pi} $};
\node (RB) at (3,-1){$J^r_{\pi}(U_i\otimes_{R}R_{\pi})$};

\draw[->] (LT) edge node[above]{$\tau_{i,p}$} (RT);
\draw[->] (LT) edge (LB);
\draw[->] (RT) edge (RB);
\draw[->] (LB) edge node[above]{$\tau_{i,\pi}$} (RB);

\end{tikzpicture}
\end{center}

where 
$$N^r_{\pi}:=\text{Ker}(J^r_{\pi}(A\otimes_{R}R_{\pi})\ra \widehat{A}\otimes_{R}R_{\pi})$$ and $\tau_{i,p},\tau_{i,\pi}$  isomorphisms.
This ends the proof.

 \qed
 
 \begin{remark}
 \label{bel}
 The exact sequence \ref{pairofmat} shows that $\text{Hom}(J^r(A),\widehat{{\mathbb G}_a})$ has rank between $(r-1)g$ and $rg$.
 \end{remark}
 
 \begin{remark}
 The above proof shows that for $\psi$ the basic $\d$-character of an elliptic curve
 one can take $\kappa=2$ in our Theorem.
 \end{remark}

\begin{definition}
\label{pdiv}
Let $A$ be a commutative group scheme over $R$ (with composition law written additively), and let $P\in A(R^{\text{alg}})$.
For any $n\geq 1$ set 
$$e_n=\min \{e(\pi); \pi\in \Pi \text{ and there exists } Q_n\in A(R_{\pi}) \textrm{ such that } p^nQ_n=P\}$$ 
and call $(e_n)$ the {\it ramification sequence} of $P$. 
Let us say that $(e_n)$ is {\it slowly growing} if the following condition is satisfied:
\begin{equation}
\label{furcht}
\inf \left\{\frac{e_n}{p^n};n\geq 1\right\}=0.\end{equation}
 \end{definition}

\begin{remark}
\ 

1) Recall that by Serre-Tate theory, cf.  \cite[p. 331]{char}, if $A$ is an Abelian scheme with ordinary reduction and no Frobenius lift then to each basis of the Tate module of $A$ mod $p$ one can attach a non-zero point $0\neq P\in A(R)$, reducing to the origin mod $p$ (hence $P$ is non-torsion), such that 
$$P\in \bigcap_{n=1}^{\infty}p^nA(R);$$
hence 
$P\in A(R^{\text{alg}})$ has ramification sequence  $(e_n)$, $e_n=0$, which is, of course, slowly growing. It would be interesting to 
have a description of all points in 
$A(R^{\text{alg}})$ with slowly  growing ramification sequence.

2) Assume $A={\mathbb G}_m$. 
Let $\alpha\in \bZ_p$ satisfy $\alpha\equiv 1$ mod $p$ and $\alpha \not\equiv 1$ mod $p^2$ and let $P=\alpha\in R^{\times}={\mathbb G}_m(R)$. We claim that $P$ has ramification sequence $(e_n)$ with $e_n\geq p^n$ (which is, of course, not slowly growing). Indeed   let $\beta_n\in (R^{\text{alg}})^{\times}$ satisfy $\beta_n^{p^n}=\alpha$. Note that $\beta_n-1$ is 
a root of the Eisenstein polynomial 
$$(x+1)^{p^n}-\alpha\in \bZ_p[x]$$
 so if $\pi_n\in \Pi$ is such that $Q_n:=\beta_n\in {\mathbb G}_m(R_{\pi_n})$ then $p^nQ_n=P$ (in additive notation) and 
 $\pi_n$ has ramification index $\geq p^n$ 
  which proves our claim. Of course, for $\psi$ the basic $\d$-character of ${\mathbb G}_m$,  we have $\psi^{\text{alg}}(P)\neq 0$.
\end{remark}

On the other hand we have the following consequence of Theorem \ref{overpsi}.

\begin{corollary}\label{slsl}
Let $A$ be an abelian scheme over $R$  and let $\psi$ be a $\d$-character of $A$.
Assume  $P\in A(R^{\text{alg}})$ has a slowly growing ramification sequence.
Then $\psi^{\text{alg}}(P)=0$.
\end{corollary}

{\it Proof}. Let $P$ have ramification sequence $(e_n)$. Assume $Q_n\in A(R_{\pi_n})$, $\pi_n\in \Pi$, $p^nQ_n=P$, $e(\pi_n)=e_n$. By Theorem \ref{overpsi} there exists $\kappa\geq 0$ such that
$$\psi^{\text{alg}}(A(R_{\pi_n}))\subset p^{-[\frac{\log e_n}{\log p}+\kappa]} R_{\pi_n}
$$
hence
$$\psi^{\text{alg}}(P)=p^n\psi^{\text{alg}}(Q_n)\in  p^{n-[\frac{\log e_n}{\log p}+\kappa]} R_{\pi_n}.$$
We get
$$v_p(\psi^{\text{alg}}(P))\geq n-\frac{\log e_n}{\log p}-\kappa$$
hence, by \ref{furcht},  $\psi^{\text{alg}}(P)=0$.
\qed

\begin{remark}\label{tito}
It would be interesting to compute 
 the intersection of the kernels of all $\psi^{\text{alg}}:A(R^{\text{alg}})\ra K^{\text{alg}}$ for $A$ an abelian scheme over $R$ where $\psi$ runs through the set of $\d$-characters of $A$. 
  This kernel contains the division hull of the group generated by all points that have a slowly growing ramification sequence.
\end{remark}

As for $\d$-modular forms we have the following extension property.

\begin{proposition}
\label{lockedin}
For any $\d$-character $\psi:A(R)\ra R$ of an abelian scheme the function
$$\psi^{\text{alg}}:A(R^{\text{alg}})\ra K^{\text{alg}}$$
extends to a  $\d^{{\mathbb C}_p}$-function
$$\psi^{{\mathbb C}_p}:A(\OCp)\ra {\mathbb C}_p.$$
\end{proposition}

{\it Proof}.
Since $\psi^{\text{alg}}$ is a homomorphism it is enough to show that $\psi^{\text{alg}}$
restricted to the unit ball ${\mathbb B}(A,\overline{0},R^{\text{alg}})$
centered at the origin $\overline{0}\in A(k)$ can be extended to a ${\mathbb C}_p$-valued continuous function on the unit ball ${\mathbb B}(A,\overline{0},\OCp)$.
But this follows from the fact, cf.  \cite[Lem. 2.8]{char}, Lemma 2.8, that $\psi$ on  
${\mathbb B}(A,\overline{0},R)$ has the form
$$\psi(P)=\Lambda(\frac{1}{p} l(t(P)))$$
for some $\Lambda=\sum_{i=1}^r \lambda_i\phi^i\in R[\phi]$, where
$l$ is the logarithm of the formal group of $A$ with respect to the formal tuple of parameters $t$ and $t(P)$ is the image of $t$ under the map $R[[t]]\ra R$ defined by $P$.
\qed

\subsection{Total $\d$-overconvergence of the forms $f^{\sharp}$}\label{sasedoi}
We first recall the construction of the $\d$-modular forms $f^{\sharp}$ attached to newforms on $\Gamma_0(N)$ given in \cite{eigen}. As usual we let $N\geq 4$, $(N,p)=1$.
 Fix, in what follows, a normalized newform
$f$
 of weight $2$ on $\Gamma_0(N)$ over ${\mathbb Q}$ and an  elliptic curve $A$ over ${\mathbb Q}$ of conductor $N$ such that $f$ and $A$ correspond to each other  so there exists
 a morphism
\begin{equation}
\label{shimuraa}
\Phi:X_0(N)\ra A
 \end{equation}
 over ${\mathbb Q}$  such that the pull back to $X_0(N)$ of some $1$-form on $A$ over ${\mathbb Q}$
corresponds to $f$.
Let $A_{R}$ be the N\'{e}ron model of $A\otimes_{\mathbb Q} K$ over $R$ (which is an elliptic curve). By the N\'{e}ron model property there is an induced morphism $\Phi:X_1(N)_{R}\ra A_{R}$. Let $X \subset X_1(N)_{R}$ be any Zariski open set (which may be the whole of $X_1(N)$). 
Let $r$ be $1$ or $2$ according as $A_{R}$ has or has not a  Frobenius lift. The image of the canonical $\d$-character $\psi
\in \cO(J^r(A_{R}))$
in (\ref{deltacharacter}) (respectively (\ref{woo})) via the map
$$\cO(J^r(A_{R})) \stackrel{\Phi^*}{\longrightarrow} \cO(J^r(X))=\cO^r(X)$$
is denoted by 
$$f^{\sharp}\in \cO^r(X).$$
If $X$ is an affine open set disjoint from the cusps then $f^{\sharp}\in M^r_X(0)$ so $f^{\sharp}$
  is a $\d$-modular form on $X$ of weight $0$ and played a key role in \cite[Sec. 3.6]{local}.
By Theorem \ref{overpsi} and Proposition \ref{lockedin} we get, for such an $X$:

\begin{corollary}
\label{onherside}
The $\d$-function 
$f^{\sharp}$ on $X$ is totally $\d$-overconvergent  with polar order bounded by the function $\lambda(x)=\frac{\log\ x}{\log\ p}+2$. In particular, $f^{\sharp}$ is tempered. Moreover $(f^{\sharp})^{\text{alg}}$ extends to a $\d^{{\mathbb C}_p}$-function
$$(f^{\sharp})^{{\mathbb C}_p}:X(\OCp)\ra {\mathbb C}_p.$$
\end{corollary}

\bibliographystyle{amsplain}

\end{document}